\documentclass[11pt]{amsart}
\usepackage{amsfonts}

\usepackage{amsmath,amssymb,amsthm}\allowdisplaybreaks[4]
\usepackage{url}
\usepackage{graphicx}   % Optional: allows for the easy inclusion of pictures.
\usepackage{verbatim}       %Package used for the comment environment%
\usepackage{layout}
\usepackage{galois}
\usepackage{mathrsfs}

\newtheorem{theorem}{Theorem}[section]
\newtheorem{lemma}[theorem]{Lemma}
\newtheorem{proposition}[theorem]{Proposition}
\newtheorem{remark}[theorem]{Remark}

\newtheorem{corollary}[theorem]{Corollary}

\newtheorem{example}[theorem]{Example}

\newfam\msbmfam\font\tenmsbm=msbm10\textfont
\msbmfam=\tenmsbm\font\sevenmsbm=msbm7
\scriptfont\msbmfam=\sevenmsbm

\def\EE{\mathbb{E}}

\def\RR{\mathbb{R}}

\def\bR{\mathbb{R}}

\def\cP{\mathcal{P}}

\def\al{{\alpha}}\def\be{{\beta}}
\def\ep{{\epsilon}}\def\ga{{\gamma}}
\def\la{{\lambda}}

\def\La{{\Lambda}}

\def\<{\left<}\def\>{\right>}
\def\({\left(}\def\){\right)}

\def\goto{{\rightarrow}}

\def\blemma{\begin{lemma}}\def\elemma{\end{lemma}}
 \def\bproposition{\begin{prop}}\def\eproposition{\end{prop}}
 \def\btheorem{\begin{theorem}}\def\etheorem{\end{theorem}}
 \def\bcorollary{\begin{corollary}}\def\ecorollary{\end{corollary}}

\def\beqlb{\begin{eqnarray}}\def\eeqlb{\end{eqnarray}}
 \def\beqnn{\begin{eqnarray*}}\def\eeqnn{\end{eqnarray*}}

\graphicspath{%
    {converted_graphics/}% inserted by PCTeX
    {/}% inserted by PCTeX
}
\begin{document}
\title
{Some support properties for \\a class of $\La$-Fleming-Viot processes}
\thanks{The research is supported by NSERC}
\author{Huili Liu and Xiaowen Zhou}
\address{Huili Liu: Department of Mathematics and Statistics, Concordia University,
1455 de Maisonneuve Blvd. West, Montreal, Quebec, H3G 1M8, Canada}
\email{l$\_$huili@live.concordia.ca}
\address{Xiaowen Zhou: Department of Mathematics and Statistics, Concordia University,
1455 de Maisonneuve Blvd. West, Montreal, Quebec, H3G 1M8, Canada}
\email{xiaowen.zhou@concordia.ca } 
\date{\today}
\begin{abstract}
For a class of $\La$-Fleming-Viot processes with underlying Brownian
motion whose associated $\La$-coalescents come down from infinity,
we prove a one-sided modulus of continuity result  for their ancestry
processes recovered from the lookdown construction of Donnelly and
Kurtz. As applications, we first show that such a $\La$-Fleming-Viot
support process has one-sided modulus of continuity (with modulus
function $C\sqrt{t\log\(1/t\)}$) at any fixed time. We also show
that the support is compact simultaneously at all positive times,
and given the initial compactness, its range is uniformly compact
over any finite time interval. In addition, under a mild condition
on the $\Lambda$-coalescence rates, we find a uniform upper bound on
Hausdorff dimension of the support and an upper bound on Hausdorff
dimension of the range.
\end{abstract}
\maketitle \pagestyle{myheadings} \markboth{\textsc{Support properties for $\La$-Fleming-Viot processes}} {\textsc{H. L. Liu
and X. W. Zhou}}
%%%%%%%%%%%%%%%%%%%%%%%%%%%%%%%%%%%%%%%%%%%%%%%%%%%%%%%%%%%%%%%%%%%%%%%%%%%%%%%%%Section
%%%%%%%%%%%%%%%%%%%%%%%%%%%%%%%%%%%%%%%%%%%%%%%%%%%%%%%%%%%%%%%%%%%%%%%%%%%%%
\section{Introduction}
Fleming-Viot process arises as a probability-measure-valued
stochastic process on the distribution of allelic frequencies in a
selectively neutral population with mutation. We refer to Ethier and
Kurtz \cite{EtKu93} and Etheridge \cite{Eth12} for surveys on the
Fleming-Viot process and related mathematical models from population
genetics.

Moments of the classical Fleming-Viot process can be expressed in
terms of a dual process involving Kingman's coalescent and semigroup
for the mutation operator. The $\La$-Fleming-Viot process
generalizes the classical Fleming-Viot process by replacing
Kingman's coalescent with the $\La$-coalescent allowing multiple
collisions. Formally, the $\La$-Fleming-Viot process is a
Fleming-Viot process with general reproduction mechanism so that the
total number of children from a parent can be comparable to the size
of population. We refer to Birkner et al. \cite{BBCEMS05} for a
connection between mutationless $\La$-Fleming-Viot processes and
continuous state branching processes. In this paper we only consider
the Fleming-Viot process with Brownian mutation that can also be
interpreted as underlying spatial Brownian motion.

The support properties are interesting in the study of
measure-valued processes. For the Dawson-Watanabe superBrownian
motion arising as high density limit of empirical measures for near
critical branching Brownian motions,  the modulus of continuity and
the carrying dimensions have been studied systematically for its
support process. We refer to Chapter 7 of Dawson \cite{Dawson1},
Chapter 9 of Dawson \cite{Dawson2} and Chapter III of Perkins
\cite{perkin} and references therein for a collection of these
results. The proofs involve the historically cluster representation,
the Palm distribution for the canonical measure and estimates
obtained from PDE associated with the Laplace functional. For a
superBrownian motion with a general branching mechanism, Delmas
\cite{Delmas} obtained the Hausdorff dimension for its range using
Brownian snake representation with subordination.

%some path properties of  including the Hausdorff dimensions for the
%supports and estimations on the hitting probability of small balls.

%It has also been shown in \cite{Delmas} that the random measure of a
%super $\alpha$-stable process with a general branching mechanism is
%absolutely continuous with respect to the Lebesgue measure.

However, the approaches for Dawson-Watanabe superBrownian motions do
not always apply to  Fleming-Viot processes which are not infinitely
divisible. Consequently, there are only a few results available for
Fleming-Viot support processes so far. The earliest work on the
compact support property for classical Fleming-Viot processes is due
to Dawson and Hochberg \cite{Dawson}. It was shown in \cite{Dawson}
that at any fixed time $T>0$ the classical Fleming-Viot process with
underlying Brownian motion has a compact support with  Hausdorff
dimension not greater than two. Using non-standard techniques
Reimers \cite{Rei} improved the above result by  showing that the
carrying dimension of the support is at most two simultaneously for
all positive times. Applying a generalized Perkins disintegration theorem,
the support dimension was found in Ruscher \cite{Ru09} for a
Fleming-Viot-like process obtained from mass normalization and time
change of superBrownian motion with stable branching.
The $\Lambda$-Fleming-Viot process does not have a compact support if the associated $\Lambda$-coalescent does not come down from infinity. Liu and Zhou
\cite{LZ} recently extended the results in \cite{Dawson} to a class
of $\Lambda$-Fleming-Viot processes whose associated
$\Lambda$-coalescents come down from infinity.  We are not aware of
any results on the modulus of continuity for Fleming-Viot support
processes although the modulus of continuity for superBrownian
motion support had been first recovered by Dawson et al. \cite{DIP89} more than twenty years ago
and further studied in Dawson and Vinogradov \cite{Dawson94} and in
Dawson el al. \cite{DHV96}.

The lookdown construction of Donnelly and Kurtz \cite{DK96,DK99a,
DK99b} is a powerful technique in the study of the Fleming-Viot
processes. Loosely speaking, the idea of lookdown construction is a
discrete representation that leads to a nice version of the
corresponding measure-valued process. The lookdown construction
naturally results in a genealogy process describing the genealogical
structure of the particles involved. In a sense it plays the role of
historical processes for Dawson-Watanabe superprocesses.

Donnelly and Kurtz \cite{DK96} first proposed the lookdown
construction of a system of countable particles embedded into the
classical Fleming-Viot process. They showed the duality between
classical Fleming-Viot process and Kingman's coalescent and
recovered some previous results on the classical Fleming-Viot
process using this explicit representation. This representation was
later extended  in Donnelly and Kurtz \cite{DK99b} via a modified
lookdown construction to a larger class of measure-valued processes
including the $\La$-Fleming-Viot processes and  the Dawson-Watanabe
superprocesses. Donnelly and Kurtz \cite{DK99a} also found a discrete representation for the classical Fleming-Viot models with selection and recombination.

Birkner and Blath \cite{BB} further discussed the modified
lookdown construction in \cite{DK99b} for the $\La$-Fleming-Viot
process with jump type mutation operator. They also described how to
recover the $\La$-coalescent from the modified lookdown
construction.

For the $\Xi$-coalescent allowing simultaneous multiple collisions,
a Poisson point process construction of the $\Xi$-lookdown model can
be found in Birkner et al. \cite{MJM} by extending the modified
lookdown construction of Donnelly and Kurtz \cite{DK99b}. It was
proved in \cite{MJM} that the empirical measure of the  exchangeable
particles converges almost surely in the Skorohod space of
measure-valued paths to the so called $\Xi$-Fleming-Viot process
with jump type mutation.

Using the modified lookdown construction of Donnelly and Kurtz, Liu
and Zhou \cite{LZ} proved that a class of $\La$-Fleming-Viot
processes with underlying Brownian motion have compact supports at
any fixed time $T>0$ provided the associated $\La$-coalescents come
down from infinity fast enough. Further,  both lower and upper
bounds were found in \cite{LZ} on Hausdorff dimension for support of
the $\La$-Fleming-Viot process at the time $T$, where the exact Hausdorff
dimension was shown to be two whenever the associated
$\Lambda$-coalescent has a nontrivial Kingman component. These
results generalize the previous results of Dawson and Hochberg
\cite{Dawson} on the classical Fleming-Viot processes.

In this paper, for the class of $\La$-Fleming-Viot processes in
\cite{LZ}, we refine the arguments in \cite{LZ} to further study
their support properties. Our first result is a one-sided modulus of
continuity type result for the ancestry process defined via the
lookdown construction. The second  result is the one-sided modulus
of continuity for the $\La$-Fleming-Viot support process at any
fixed time. The third is on the uniform compactness of the
$\La$-Fleming-Viot support and the associated range.  Under an
additional mild condition on the coalescence rates of the
corresponding $\La$-coalescent, we also obtain two results on
support dimensions. One result is an uniform upper bound on
Hausdorff dimension for the support at all positive times. The other
is an upper bound on Hausdorff dimension for the range of the
$\La$-Fleming-Viot support process. Again, the lookdown construction
plays a key role throughout our arguments.

The paper is arranged as follows. In Section \ref{s2} we introduce
the $\La$-coalescent and the corresponding coming down from infinity
property. In Section \ref{s3} we briefly discuss the lookdown
construction for $\La$-Fleming-Viot process with underlying Brownian
motion and the associated ancestry process recovered from the
lookdown construction. In Section \ref{s4} we present the main
results of this paper together with corollaries and propositions. Proofs of the main results are deferred to Section
\ref{s5}.

%%%%%%%%%%%%%%%%%%%%%%%%%%%%%%%%%%%%%%%%%%%%%%%%%%%%%%%%%%%%%%%%%%%%%%%
\section{The $\La$-coalescent}\label{s2}

\subsection{The $\La$-coalescent}
We first introduce some notation. Put $[n]\equiv\{1,\ldots,n\}$ and
$[\infty]\equiv\{1,2,\ldots\}$. An ordered {\it partition} of $D\subset
[\infty]$ is a countable collection $\pi=\{\pi_i, i=1,2,\ldots\}$ of
disjoint {\it blocks} such that $\cup_{i}\pi_i=D$ and
$\min\pi_i<\min\pi_j$ for $i<j$. Then blocks in $\pi$ are ordered by
their least elements.

Denote by $\mathcal{P}_n$ the set of ordered partitions of $[n]$ and
by $\mathcal{P}_\infty$  the set of ordered partitions of
$[\infty]$. Write
$\mathbf{0}_{[n]}\equiv\left\{\{1\},\ldots,\{n\}\right\}$ for the
partition of $[n]$ consisting of singletons and
$\mathbf{0}_{[\infty]}$ for the partition of $[\infty]$ consisting
of singletons. Given $n\in[\infty]$ and $\pi\in\cP_\infty$, let
$R_n(\pi)\in\cP_n$ be the restriction of $\pi$ to $[n]$.

Kingman's coalescent is a $\mathcal{P}_{\infty}$-valued time
homogeneous Markov process such that all different pairs of blocks
independently merge at the same rate. Pitman \cite{Pitman99} and
Sagitov \cite{Sag99} generalized the Kingman's coalescent to the
{\it $\La$-coalescent} which allows {\it multiple collisions}, i.e.,
more than two blocks may merge at a time. The $\La$-coalescent is
defined as a $\mathcal{P}_{\infty}$-valued Markov process $\Pi\equiv\(\Pi(t)\)_{t\geq 0}$ such that for each $n\in[\infty]$, its restriction to
$[n]$, $\Pi_n\equiv\(\Pi_n(t)\)_{t\geq 0}$ is a $\mathcal{P}_n$-valued Markov
process whose transition rates are described as follows: if there
are currently $b$ blocks in the partition, then each $k$-tuple of
blocks ($2\leq k\leq b$) independently merges to form a single block
at rate
\begin{eqnarray}\label{la_rate}
\la_{b,k}=\int_{[0,1]}x^{k-2}(1-x)^{b-k}\La(dx),
\end{eqnarray}
where $\La$ is a finite measure on $[0,1]$. It is easy to check that
the rates $\(\la_{b,k}\)$ are consistent so that for all $2\leq
k\leq b$,
\begin{equation}
\begin{split}\label{20120607eq1}
\la_{b,k}=\la_{b+1,k}+\la_{b+1, k+1}.
\end{split}
\end{equation}
Consequently, for any $1\leq m<n\leq\infty$, the coalescent process
$R_m\(\Pi_n(t)\)$ given $\Pi_n(0)=\pi_n$ has the same distribution
as $\Pi_m(t)$ given $\Pi_m(0)=R_m(\pi_n)$.

With the transition rates determined by (\ref{la_rate}), there
exists a one to one correspondence between $\La$-coalescents and
finite measures $\La$ on $[0,1]$.

For $n=2,3,\ldots$, denote by
\begin{equation}\label{la_n}
\la_n=\sum_{k=2}^n{n\choose k}\la_{n,k}
\end{equation}
the total coalescence
rate starting with $n$ blocks. It is clear that $\(\la_n\)_{n\geq2}$ is an increasing sequence, i.e., $\la_n\leq\la_{n+1}$ for any $n\geq2$. In addition, denote by
$$\gamma_n=\sum_{k=2}^n\(k-1\){n\choose k}\la_{n,k}$$
the rate at which the number of blocks decreases.

\subsection{Coming down from infinity}
Given any $\La$-coalescent $\Pi\equiv\(\Pi(t)\)_{t\geq 0}$ with $\Pi(0)=\mathbf{0}_{[\infty]}$, let $\#\Pi(t)$ be the number of blocks in the partition $\Pi(t)$. The $\La$-coalescent $\Pi$ {\it comes down
from infinity} if $$\mathbb{P}\(\#\Pi(t)<\infty\)=1$$ for all $t>0$ and it  {\it stays infinite} if $$\mathbb{P}\(\#\Pi(t)=\infty\)=1$$ for all $t>0$.
Suppose that the measure $\La$ has no atom at $1$.
It is shown by Schweinsberg \cite{Jason2000} that %\label{801}
\begin{itemize}
\item the $\La$-coalescent comes down from infinity if and only if
$\sum_{n=2}^{\infty}\gamma_n^{-1}<\infty$;
\item the
$\La$-coalescent stays infinite if and only if
$\sum_{n=2}^{\infty}\gamma_n^{-1}=\infty$.
\end{itemize}
It is pointed out in Bertoin and Le Gall \cite{BerLeGa06} that for
\[\psi(q)=\int_{[0,1]}(e^{-qx}-1+qx)x^{-2}\Lambda(dx),\]
\[\sum_{n=2}^{\infty}\gamma_n^{-1}<\infty \text{\,\, if and only if \,\,} \int_a^\infty \frac{1}{\psi(q)}dq<\infty,\]
where the integral is finite for some (and then for all) $a>0$.

\begin{example}
In case of $\La=\delta_0$, the corresponding coalescent is Kingman's
coalescent and comes down from infinity.
\end{example}

\begin{example}
If $\be\in(0, 2)$ and $$\La(dx)=\frac{\Gamma(2)}{\Gamma(2-\be)\Gamma(\be)}x^{1-\be}\(1-x\)^{\be-1}dx,$$
the corresponding coalescent is {\it Beta$(2-\be,\be)$-coalescent}.
\begin{itemize}
\item In case of $\be\in\left(0,1\right]$, it stays infinite.
\item In case of $\be\in\left(1,2\right)$, it comes down from infinity.
\end{itemize}
\end{example}
%%%%%%%%%%%%%%%%%%%%%%%%%%%%%%%%%%%%%section%%%%%%%%%%%%%%%%%%%%%%%%%%%%%%%%%%%%%%%%%%%%%%%%%%%%%%
%%%%%%%%%%%%%%%%%%%%%%%%%%%%%%%%%%%%%%%%%%%%%%%%%%%%%%%%%%%%%%%%%%%%%%%%%%%%%%%%%%%%%%%%%%%
\section{The $\La$-Fleming-Viot process and its lookdown construction}\label{s3}
In this section, we first discuss the lookdown construction of
$\La$-Fleming-Viot process with underlying Brownian motion. Then we explain
how to recover the $\La$-coalescent from the lookdown construction. Finally, we introduce the ancestry process for the $\La$-Fleming-Viot process from the lookdown construction.

\subsection{Lookdown construction of $\La$-Fleming-Viot process with underlying Brownian motion}
Donnelly and Kurtz \cite{DK99b} introduced a modified lookdown
construction with the empirical measure process converging to
measure-valued stochastic process. A key advantage of the lookdown
construction is its projective property. Intuitively, in the lookdown model each
particle is attached a ``level" from the set $\{1,2,\ldots\}$. The
evolution of a particle at level $n$ only depends on the evolution
of the finite particles at lower levels. This property allows us to
construct approximating particle systems, and their limit as
$n\rightarrow\infty$ in the same probability space.

Following Birkner and Blath \cite{BB}, we now give a brief
introduction on the modified lookdown construction of the $\Lambda$-Fleming-Viot process with underlying Brownian motion. Let
$$\(X_1(t),X_2(t),X_3(t),\ldots\)$$ be an $(\RR^d)^{\infty}$-valued random variable, where for any $i\in[\infty]$, $X_i(t)$
represents the spatial location of the particle at level $i$. We
require the initial values $\left\{X_i(0),i\in[\infty]\right\}$ to
be exchangeable random variables so that the limiting empirical
measure
$$\lim_{n\rightarrow\infty}\frac{1}{n}\sum_{i=1}^n\delta_{X_i(0)}$$
exists almost surely by de Finetti's theorem.

Let $\La$ be the finite measure associated to the $\La$-coalescent.
The reproduction in the particle system consists of two kinds of
birth events: the events of single birth  determined by measure
$\La(\{0\})\delta_0$ and the events of multiple births determined by
measure $\La$ restricted to $(0,1]$ that is denoted by
$\La_0$.

To describe the evolution of the system during events of single
birth, let $\{\mathbf{N}_{ij}(t): 1\leq i<j<\infty\}$ be independent
Poisson processes with common rate $\La(\{0\})$. At a jump time $t$ of
$\mathbf{N}_{ij}$, the particle at level $j$ looks down at the
particle at level $i$ and
 assumes its location (therefore, particle at level $i$ gives birth to a new particle). Values of  particles at levels
above $j$ are shifted accordingly, i.e., for $\Delta
{\mathbf{N}}_{ij}(t)=1$, we have
\begin{eqnarray}
X_k(t)=\begin{cases} X_k(t-), &\text{~~if $k<j$},\\
X_i(t-), &\text{~~if $k=j$},\\
X_{k-1}(t-), &\text{~~if $k>j$}.
\end{cases}
\end{eqnarray}

For those events of multiple births we can construct an independent
Poisson point process $\tilde{\mathbf{N}}$ on $\mathbb{R}^{+}\times\left(0,
1\right]$ with intensity measure $dt\otimes x^{-2}\La_0\(dx\)$. Let
$\{U_{ij}$, $i, j\in[\infty]\}$ be i.i.d. uniform $[0, 1]$ random
variables. Jump points $\{\(t_i, x_i\)\}$ for $\tilde{\mathbf{N}}$ correspond to
the multiple birth events. For $t\geq 0$ and  $J\subset [n]$ with $|J|\geq 2$,
define
\begin{eqnarray}\label{times}
\mathbf{N}_J^n(t)\equiv\sum_{i:t_i\leq t}\prod_{j\in
J}\mathbf{1}_{\{U_{ij}\leq x_i\}}\prod_{j\in[n]\backslash
J}\mathbf{1}_{\{U_{ij}> x_i\}}.
\end{eqnarray}
Then $\mathbf{N}_J^n(t)$ counts the number of birth events among the particles from
levels $\{1,2,\ldots,n\}$ such that exactly those at levels in $J$ are involved up to
time $t$. Intuitively, at a jump time $t_i$, a uniform coin is
tossed independently for each level. All the particles at levels $j$
with $U_{ij}\leq x_i$ participate in the lookdown event. More
precisely, those particles involved jump to the location of the
particle at the lowest level involved. The spatial locations of
particles on the other levels, keeping their original order, are
shifted upwards accordingly, i.e., if $t=t_i$ is the jump time and
$j$ is the lowest level involved, then
\begin{eqnarray*}
X_k(t)=
\begin{cases}
X_k(t-), \text{~for~} k\leq j,\\
X_j(t-), \text{~for~} k>j \text{~with~} U_{ik}\leq x_i,\\
X_{k-J_t^k}(t-), \text{~otherwise},
\end{cases}
\end{eqnarray*}
where $J_{t_i}^k\equiv\#\{m<k, U_{im}\leq x_i\}-1$.

Between jump times of the Poisson processes, particles at different
levels move independently according to Brownian motions in $\bR^d$.

We assume that  the above-mentioned
lookdown construction is carried out in a probability space $(\Omega, \mathcal{F}, \mathbb{P})$ .

For each $t>0$, $X_1(t),X_2(t),\ldots$ are known to be exchangeable
random variables so that
$$X(t)\equiv\lim_{n\rightarrow\infty}X^{(n)}(t)\equiv \lim_{n\rightarrow\infty}\frac{1}{n}\sum_{i=1}^n\delta_{X_i(t)}$$
exists almost surely by  de Finetti's theorem and follows the
probability law of the $\La$-Fleming-Viot process with underlying
Brownian motion. Further, we have that $X^{\(n\)}$ converges to $X$
in the path space $D_{M_1\(\RR^d\)}([0,\infty))$ equipped with the
Skorohod topology, where $M_1\(\RR^d\)$ denotes the space of
probability measures on $\RR^d$ equipped with the topology of weak
convergence. See Theorem 3.2 of \cite{DK99b}.

In the sequel we
always write $X$ for such a $\La$-Fleming-Viot process.
Write supp\,$\mu$ for the closed support of measure $\mu$.

\begin{lemma}\label{support}
For any  $t\geq 0$, $\mathbb{P}$-a.s.  the spatial locations
of the countably many particles in the lookdown construction satisfy
$$\left\{X_1(t),X_2(t),X_3(t),\ldots\right\}\subseteq\text{supp}\,X(t).$$
\end{lemma}

\begin{proof}
In the lookdown construction, $\(X_n(t)\)_{n\geq 1}$ are
exchangeable at any time $t\geq0$. By de Finetti's theorem (cf.
Aldous \cite{Aldous}) such a system is a mixture of
i.i.d. sequence, i.e., given the empirical measure
$$X(t)=\lim_{n\rightarrow\infty}\frac{1}{n}\sum_{i=1}^{n}\delta_{X_i(t)},$$
 the random variables $\{X_i(t), i=1,2,\ldots\}$ are
 jointly distributed as
 i.i.d. samples from the directing measure $X(t)$. Therefore,
$X_n(t)\in  \text{supp}\,X(t)$ for any $n\in[\infty]$.
\end{proof}

\subsection{The $\La$-coalescent in the lookdown construction}
The birth events induce a family structure to the particle system so
we can present the {\it genealogy process} first introduced in
Donnelly and Kurtz \cite{DK99b}. For any $0\leq t\leq s$ and
$n\in[\infty]$, denote by $L_n^s(t)$ the ancestor's level at time
$t$  for the particle with level $n$ at time $s$. Given $s$ and $n$,
$L_n^s(t)$ is nondecreasing and left continuous in $t$. Moreover,
the genealogy processes $(L_n^s)_{s\geq 0}, n=1,2,\ldots$ satisfy
the equations
\begin{eqnarray*}
\begin{split}
L_n^s(t)=n&-\sum_{1\leq i<j<n}\int_{t-}^s\mathbf{1}_{\{L_n^s(u)>j\}}d\mathbf{N}_{ij}(u)\\
&-\sum_{1\leq i<j\leq n}\int_{t-}^s\(j-i\)\mathbf{1}_{\{L_n^s(u)=j\}}d\mathbf{N}_{ij}(u)\\
&-\sum_{J\subset
[n]}\int_{t-}^s\(L_n^s\(u\)-\min J\)\mathbf{1}_{\{L_n^s\(u\)\in
J\}}d\mathbf{N}_J^n(u)\\
&-\sum_{J\subset
[n]}\int_{t-}^s\(|J\cap\{1,\ldots,L_n^s\(u\)\}|-1\)\times
\mathbf{1}_{\{L_n^s\(u\)>\min J, L_n^s\(u\)\not\in
J\}}d\mathbf{N}_J^n\(u\).
\end{split}
\end{eqnarray*}
Given $T>0$, for any $0\leq t\leq T$ and $i\in[\infty]$, $L_i^T\(T-t\)$ represents the
ancestor's level at time $T-t$ of the particle with level $i$ at
time $T$ and $X_{L_i^T\(T-t\)}\((T-t)-\)$ represents that ancestor's
location.

Write $\big(\Pi^T(t)\big)_{0\leq t\leq T}$ for the
$\mathcal{P}_\infty$-valued process such that $i$ and $j$ belong to
the same block of $\Pi^T(t)$ if and only if $L_i^T(T-t)=L_j^T(T-t)$,
i.e., $i$ and $j$ belong to the same block if and only if the two
particles with levels $i$ and $j$, respectively, at time $T$ share a
common ancestor at time $T-t$. The process $\big(\Pi^T(t)\big)_{
0\leq t\leq T}$ turns out to have the same law as the
$\La$-coalescent running up to time $T$. See Donnelly and Kurtz
\cite{DK99b} and Birkner and Blath \cite{BB}.

The next property of the genealogy process can be found in Lemma 3.1
of \cite{LZ}.
\begin{lemma}\label{level}
For any fixed $T>0$, let $\(\Pi^{T}(t)\)_{0\leq t\leq T}$ be
the $\La$-coalescent recovered from the lookdown
construction. Then given $t\in[0,T]$ and the ordered random
partition $\Pi^T(t)=\left\{\pi_l(t):
l=1,\ldots,\#\Pi^T(t)\right\}$, we have
$$L_j^T\(T-t\)=l\text{~~for any~~}j\in\pi_l(t).$$
\end{lemma}

\subsection{Ancestry process} For any $T>0$,
denote by
$$\(X_{1,s},X_{2,s},X_{3,s},\ldots\)_{0\leq s\leq T}$$ the  {\it ancestry process}
with $X_{i,s}$ defined by
\begin{equation}
X_{i,s}\(t\)\equiv
X_{L_i^s\(t\)}\(t-\) \text{~~for~~}0\leq t\leq s.\\
\end{equation}
Intuitively $X_{i,s}$ keeps track  of locations for all the
ancestors of the particle with level $i$ at time $s$.

For any $s\geq 0$, we can recover the $\La$-coalescent $\(\Pi^s(t)\)_{0\leq t\leq s}$ from the lookdown construction. For any $0\leq r<s$, set
$$N^{r,s}\equiv\#\Pi^{s}\(s-r\)$$
and
$$\Pi^{s}\(s-r\)\equiv\{\pi_l: 1\leq l\leq N^{r,s}\},$$
where $\pi_l\equiv\pi_l(r,s),\  1\leq l\leq N^{r,s}$ are all the
disjoint blocks of $\Pi^{s}\(s-r\)$ ordered by their least
elements.
Let $H(r,s)$ be the maximal dislocation between the countably many
particles at time $s$ and their respective ancestors at time $r$. Applying Lemma \ref{level}, we have
\begin{eqnarray*}
\begin{split}
H\(r,s\)\equiv&\max_{1\leq l\leq
N^{r,s}}\max_{j\in\pi_{l}}\left|X_j(s)-X_{L_j^{s}(r)}\(r-\)\right|\\
=&\max_{1\leq l\leq
N^{r,s}}\max_{j\in\pi_l}\left|X_j(s)-X_l(r-)\right|.
\end{split}
\end{eqnarray*}
%%%%%%%%%%%%%%%%%%%%%%%%%%%%%%%%%%%%%%%section%%%%%%%%%%%%%%%%%%%%%%%%%%%%%%%%%%%
%%%%%%%%%%%%%%%%%%%%%%%%%%%%%%%%%%%%%%%%%%%%%%%%%%%%%%%%%%%%%%%%%%%%%%%%%%%%%%%%%%%%%%%%%%
\section{Some properties of the $\La$-Fleming-Viot process}\label{s4}
\subsection{Main results}

For any $T>0$, let $(\Pi^T(t))_{ 0\leq t\leq T}$ be the $\La$-coalescent
recovered from the lookdown construction with
$\Pi^T(0)=\mathbf{0}_{[\infty]}$.
Write $\Pi\equiv(\Pi(t))_{t\geq 0}$ for the unique (in law) $\Lambda$-coalescent such that $(\Pi(t))_{ 0\leq t\leq T}$ has the same distribution as $(\Pi^T(t))_{ 0\leq t\leq T}$. We call $\Pi$ the $\Lambda$-coalescent  associated to the $\Lambda$-Fleming-Viot process $X$.

For any positive integer $m$, set
\begin{equation}\label{711}
T_m\equiv \inf\big\{t\geq0: \#\Pi(t)\leq m\big\}
\end{equation}
with the convention $\inf\emptyset=\infty$.

Given $\eta>0$, for any Borel set ${A}\subset\RR^d$, let
$\mathbb{B}\({A},\eta\)$ be its {\it closed $\eta$-neighborhood} such that
$$\mathbb{B}\({A},\eta\)\equiv\overline{\bigcup_{x\in A}\mathbb{B}\(x,\eta\)},$$
where $\mathbb{B}\(x,\eta\)$ denotes the closed ball centered at $x$
with radius $\eta$.

We now recall the definition of {\it Hausdorff dimension}. Given
$A\subset\mathbb{R}^d$ and $\be>0$, $\eta>0$, let
$${\mathcal{H}}_{\eta}^{\be}\(A\)\equiv\inf_{\{S_l\}\in{\varphi}_{\eta}}\sum_l{d\(S_l\)}^{\be},$$
where $d\(S_l\)$ denotes the diameter of ball $S_l$ in
$\mathbb{R}^d$ and $\varphi_{\eta}$ denotes the collection of {\it
$\eta$-covers} of set $A$ by balls with diameters at most $\eta$.
The Hausdorff
$\be$-measure of $A$ is defined by
$${\mathcal{H}}^{\be}(A)=\lim_{\eta\rightarrow 0}{\mathcal{H}}_{\eta}^{\be}\(A\).$$ The Hausdorff dimension of $A$ is defined
by
$$\text{dim~} A\equiv\inf\big\{\be>0: {\mathcal{H}}^{\be}\(A\)=0\big\}=\sup\big\{\be>0: {\mathcal{H}}^{\be}\(A\)=\infty\big\}.$$

Recall that $X$ is the $\La$-Fleming-Viot process with underlying Brownian motion. For any subset
$\mathcal{I}\subset\RR\cap[0,\infty)$, let
$$\mathcal{R}(\mathcal{I})\equiv\overline{\cup_{t\in
\mathcal{I}}\,\text{supp}\,X(t)}$$
be the range of supp\,$X$ on the time interval $\mathcal{I}$.

Throughout the paper, we always write $C$ or $C$ with subscript for
a positive constant and write $C(x)$
for a constant depending on $x$ whose values might vary from place to
place.
The main results of this paper are the following theorems.
% where we
%always denote $X$ for a $\La$-Fleming-Viot process with underlying
%Brownian motion in $\RR^d$.
We defer the proofs to Section \ref{s5}.

\noindent
{\bf Assumption I}:
There exists a constant $\al>0$ such that the associated $\La$-coalescent $\Pi$ satisfies $$\limsup_{m\rightarrow\infty}m^{\al}\EE T_m <\infty.$$

\begin{theorem}\label{th1}
Under Assumption I and for any $T>0$, there exist a positive random
variable $\theta\equiv\theta\(T,d,\al\)<1$ and a constant $C\equiv
C(d,\al)$ such that $\mathbb{P}$-a.s. for all $r,s\in [0,T]$
satisfying $0<s-r\leq\theta $, we have
\begin{eqnarray}\label{mod1}
\begin{split}
H\(r,s\) \leq &C\sqrt{\(s-r\)\log\(1/\(s-r\)\)}.
\end{split}
\end{eqnarray}
\end{theorem}

\begin{theorem}\label{modulus2}
Under Assumption I and given any fixed $t\geq0$, there exist a
positive random variable $\theta\equiv\theta\(t,d,\al\)<1$ and
a constant $C\equiv C(d,\al)$ such that for any $\Delta t$ with
$0<\Delta t\leq\theta$ we have $\mathbb{P}$-a.s.
\begin{eqnarray}
\begin{split}
\text{supp}\,X\(t+\Delta t\)\subseteq\mathbb{B}\(\text{supp}\,X(t), C\sqrt{\Delta t\log\(1/\Delta t\)}\).
\end{split}
\end{eqnarray}
\end{theorem}

\begin{theorem}\label{916co1}
Under Assumption I, supp\,$X(t)$ is compact
for all  $t>0$ $\mathbb{P}$-a.s.. Further,  if supp\,$X(0)$ is compact, then ${\mathcal{R}([0,t))}$ is compact for all $t>0$ $\mathbb{P}$-a.s..
\end{theorem}

\noindent {\bf Condition A}: There exists a constant $\al>0$ such that the associated $\La$-coalescent $\Pi$ satisfies
$$\limsup_{m\rightarrow\infty}m^{\al}\sum_{b=m+1}^{\infty}{\la_{b}}^{-1}<\infty.$$

\begin{remark}
The Kingman's coalescent satisfies Condition A with $\al=1$. In case of $\be\in(1,2)$, the Beta$(2-\be, \be)$-coalescent satisfies Condition A with $\al=\be-1$.
\end{remark}

\begin{theorem}\label{811t1}
Suppose that Condition A holds. Then
$$\text{dim}\,\text{supp}\,X(t)\leq 2/\al$$   for all $t>0$ $\mathbb{P}$-a.s..
\end{theorem}

\begin{theorem}\label{range}
Suppose that Condition A holds. Then for any $0<\delta<T$, $$\text{dim}\,{\mathcal{R}([\delta,T))}\leq 2+2/\al\ \ \ \mathbb{P}\text{-a.s..}$$
\end{theorem}

\subsection{A sufficient condition}

Recall the Markov chain introduced in {\cite{LZ}}. For any $n$, $\(\Pi_n(t)\)_{t\geq 0}$ is the $\La$-coalescent $\Pi$ restricted to $\left[n\right]$ with
$\Pi_n(0)=\mathbf{0}_{[n]}$. For any $n>m$, the block counting process
$\(\#\Pi_n(t)\vee m\)_{t\geq0}$ is a Markov chain
with initial value $n$ and absorbing state $m$. For any $n\geq b>m$,
let $\(\mu_{b,k}\)_{m\leq k\leq b-1}$  be its transition rates such
that
\begin{eqnarray}\label{M_rate}
\begin{cases}
&\mu_{b,b-1}={b\choose 2}\la_{b,2},\\
&\mu_{b,b-2}={b\choose 3}\la_{b,3},\\
&\cdots\cdots\\
&\mu_{b,m+1}={b\choose b-m}\la_{b,b-m},\\
&\mu_{b,m}=\sum_{k=b-m+1}^b{b\choose k}\la_{b,k}.
\end{cases}
\end{eqnarray}
The total transition rate is
$$\mu_b=\sum_{k=m}^{b-1}\mu_{b,k}=\sum_{k=2}^b{b\choose
k}\la_{b,k}=\la_b.$$

For $b>m$, let $\gamma_{b,m}$ be the total rate at which the block
counting Markov chain  starting at $b$ is decreasing, i.e.,
\begin{eqnarray}\label{1213}
\,\,\,\,\gamma_{b,m}=\begin{cases}\sum_{k=2}^{b-m}\(k-1\){b\choose
k}\la_{b,k}+\sum_{k=b-m+1}^{b}\(b-m\){b\choose k}\la_{b,k},
&\text{~~if~~} b\geq m+2,\\
\sum_{k=2}^{b}{b\choose k}\la_{b,k}, &\text{~~if~~} b=m+1.
\end{cases}
\end{eqnarray}

\noindent {\bf Condition B}: There exists a constant $\al>0$ such that
$$\limsup_{m\rightarrow\infty}m^{\al}\sum_{b=m+1}^{\infty}{\gamma_{b,m}}^{-1}<\infty.$$

\begin{remark}\label{824re}
It follows from the proof of Lemma 4.4 in \cite{LZ} that $$\EE T_m\leq\sum_{b=m+1}^{\infty}{\gamma_{b,m}}^{-1}.$$
Recalling the definitions of $\ga_{b,m}$ by (\ref{1213}) and $\la_b$ by (\ref{la_n}), we have $\la_b\leq\ga_{b,m}$ for any $b>m$. Then
for any $\al>0$, we have
$$m^{\al}\EE T_m\leq m^{\al}\sum_{b=m+1}^{\infty}{\ga_{b,m}}^{-1}\leq m^{\al}\sum_{b=m+1}^{\infty}{\la_{b}}^{-1}.$$
Therefore, Condition A implies Condition B which is sufficient for Assumption I.

Condition A is not a strong requirement since for the Beta
coalescents Condition A is sufficient and necessary  for coming down
from infinity.

The speed of coming down from infinity for $\Lambda$-coalescent is
discussed in Berestycki et al \cite{BeBeLi10}. It is shown that
there exists a deterministic function $\nu: (0, \infty)\goto
(0,\infty)$ such that $\#\Pi(t)/\nu(t)\goto 1$ as $t \goto 0$ both
almost surely and in $L^p$ for $p\geq 1$. For our purpose, it is
possible to replace Assumption I with an assumption on the behavior
of $\nu(t)$ for $t$ close to $0$.
\end{remark}

\subsection{Some Corollaries and Propositions}
For $t>0$, let $$r(t)\equiv\inf\left\{R\geq 0:\text{supp}\,X\(t\)\subseteq\mathbb{B}\(0,R\)\right\}.$$

The next result is similar to Theorem 2.1 of Tribe \cite{Tri89} on
the support process of superBrownian motion; also see Theorem
9.3.2.3 of Dawson \cite{Dawson2}.  It follows immediately  from
Theorem \ref{modulus2}.
\begin{corollary}
Under Assumption I, there
exists a constant $C>0$ such that
$$\mathbb{P}_{\delta_0}\left(\limsup_{t\downarrow0}\frac{\sup_{0\leq u\leq t}r(u)}{\sqrt{t\log\(1/t\)}}\leq C\right)=1,$$
where $\mathbb{P}_{\delta_0}$ denotes the law of $X$ with
$X\(0\)=\delta_0$.
\end{corollary}

\begin{corollary}\label{121013}
Suppose that Condition A holds. For any $T>0$, we have
$$\mathbb{P}_{\delta_0}\left(\text{dim}\,\mathcal{R}\(\left[0,T\right)\)\leq 2+2/\al\right)=1.$$
\end{corollary}
We defer the proof of Corollary \ref{121013} to Section \ref{s5}.

The next result follows from the proof of Theorem \ref{811t1} and a
standard result of Hausdorff measure; see Lemma 6.3 of Falconer
\cite{Fal}.

\begin{proposition}
Suppose that Condition A holds. Then $\mathbb{P}$-a.s. for all $t>0$
and $\epsilon>0$ we have
\[\limsup_{r\goto 0+}\frac{X(t)(\mathbb{B}(x,r))}{r^{2/\alpha+\epsilon}}>0\]
for $X(t)$ almost all $x$.
\end{proposition}

For any $0<t<1$, let
\begin{equation}\label{7112}
h(t)\equiv\sqrt{ t\log\({1}/{t}\)}.
\end{equation}

\begin{proposition}\label{4.10}
Let $X$ be any $\La$-Fleming-Viot process with $\La(\{0\})>0$ and underlying Brownian motion in $\RR^d$ for $d\geq2$.
Then given any fixed $t\geq0$, with probability one the process supp\,$X(t)$ has the one-sided
modulus of continuity with respect to $Ch$, where $C\equiv C(d)$ is the constant determined in Theorem \ref{modulus2}.
Further, with probability one supp\,$X(t)$ is compact for all $t>0$ and if supp\,$X(0)$ is compact, then $\mathcal{R}\([0,t)\)$ is also compact for all $t>0$. In addition, with probability one
$$\text{dim}\,\text{supp}\,X(t)\leq 2$$
for all $t>0$. Finally,  given any $0<\delta<T$, with probability one
$$\text{dim}\,\mathcal{R}\([\delta,T)\)\leq 4.$$
\end{proposition}

\begin{proof}
Since $\Lambda(\{0\})>0$, the $\Lambda$-coalescent has a nontrivial
Kingman component. Then
\[\la_b\geq \frac{1}{2}\Lambda(\{0\})b(b-1)\]
and
\[\sum_{b=m+1}^{\infty}\frac{1}{\la_b}\leq\sum_{b=m+1}^{\infty}\frac{2}{\Lambda(\{0\})b(b-1)}
=\frac{2}{\Lambda(\{0\})m}, \]
i.e., Condition A holds with $\al=1$. Therefore, the results follow from Remark \ref{824re} and Theorems \ref{modulus2}-\ref{range}.
\end{proof}

\begin{remark}
The uniform upper bound on the Hausdorff dimension of classical
Fleming-Viot support process
was first proved by Reimers
\cite{Rei}, where a non-standard construction of the classical Fleming-Viot process is
used to establish this result.
\end{remark}

Recall the {\it $(c,\ep,\gamma)$-property} introduced in \cite{LZ}. We say that a $\La$-coalescent has the $\(c,\epsilon,\gamma\)$-property, if
there exist constants $c>0$ and $\ep$, $\gamma\in(0,1)$ such that the measure $\La$ restricted to $[0,\ep] $  is absolutely continuous with respect to Lebesgue measure  and $${\La(dx)}\geq cx^{-\gamma}dx \text{\, for all \,} x\in [0,\ep].$$
The $\La$-coalescents with the $\(c,\epsilon,\gamma\)$-property come down from infinity.

\begin{proposition}\label{7153}
Let $X$ be any
$\La$-Fleming-Viot process with underlying Brownian motion in
$\RR^d$ for $d\geq2$. If the associated $\La$-coalescent has the $(c,\ep,\ga)$-property, then given any fixed $t\geq0$, with probability one the process supp\,$X(t)$ has the one-sided modulus of continuity with respect to $Ch$, where $C\equiv C(d,\ga)$ is the constant determined in Theorem \ref{modulus2}.
Further, with probability one supp\,$X(t)$ is compact for all $t>0$ and if supp\,$X(0)$ is compact, then $\mathcal{R}\([0,t)\)$ is also compact for all $t>0$. In addition, with probability one
$$\text{dim}\,\text{supp}\,X(t)\leq 2/\gamma$$
for all $t>0$. Finally, given any $0<\delta<T$, with probability one
$$\text{dim}\,\mathcal{R}\([\delta,T)\)\leq 2+2/\ga.$$
\end{proposition}

\begin{proof}
It has been proved by Lemma 4.13 of \cite{LZ} that
for any $n\geq 2$, there exists a positive constant $C(c,\epsilon,\gamma)$
such that the total coalescence rate of the $\La$-coalescent with the $(c,\ep,\ga)$-property
satisfies
$$\lambda_n\geq C(c,\ep,\gamma)n^{1+\gamma}.$$
Then
\begin{equation*}
\begin{split}
\sum_{b=m+1}^{\infty}\frac{1}{\la_b}\leq&\frac{1}{C(c,\ep,\ga)}\int_{m}^{\infty}\frac{1}{x^{1+\ga}}dx\leq\frac{1}{\ga
C(c,\ep,\ga)m^{\ga}},
\end{split}
\end{equation*}
i.e., Condition A holds with $\al=\ga$. Consequently, the results follow from Remark \ref{824re} and Theorems \ref{modulus2}-\ref{range}.
\end{proof}

Now we discuss the support properties for {\it Beta$(2-\be,\be)$-Fleming-Viot process} with underlying Brownian motion.
 %For $\be\in(0, 2)$, the { Beta$(2-\be,\be)$-coalescent} is
%the $\La$-coalescent with the finite measure $\La$ on $[0,1]$
%denoted by
%$$\La(dx)=\frac{\Gamma(2)}{\Gamma(2-\be)\Gamma(\be)}x^{1-\be}\(1-x\)^{\be-1}dx.$$
It is known that  the Beta$(2-\be,\be)$-coalescent stays infinite if
$\be\in\left(0,1\right]$ and comes down from infinity
 if $\be\in\left(1,2\right)$. For $\be\in(1,2)$, given any $\ep\in(0,1)$,
the Beta$\(2-\be,\be\)$-coalescent has the $\(c,\ep,\be-1\)$-property. Therefore, the conclusions of Proposition \ref{7153} hold with $\ga=\be-1$.

For $t\geq 0$ put
\[S_t\equiv\cap_{n=1}^\infty \mathcal{R}([t, t+1/n)).\]

\begin{proposition}\label{prop_mod_uni}
Under Assumption I and for any $T>0$, there exist a positive random
variable $\theta\equiv\theta\(T,d,\al\)<1$ and a constant $C\equiv
C(d,\al)$ such that $\mathbb{P}$-a.s.
\[\text{supp} X(t+\Delta t)\subseteq \mathbb{B} (S_t, Ch(\Delta t)) \]
for all $0\leq t<t+\Delta t\leq T$ and $0<\Delta t\leq\theta$.
\end{proposition}
We also defer the proof of Proposition \ref{prop_mod_uni} to Section \ref{s5}.

%%%%%%%%%%%%%%%%%%%%%%%%%%%%%%%%%%%%%%%%%%%%%%%%%%%%%%%%%%%%%%%%%%%%%%%%%%%%%%%%%%%%%%%%
%%%%%%%%%%%%%%%%%%%%%%%%%%%%%%%%%%%%%%%%%%%%%%%%%%%%%%%%%%%%%%%%%%%%%%%%%%
\section{Proofs of Theorems \ref{th1}-\ref{range}, Corollary \ref{121013} and Proposition \ref{prop_mod_uni}}\label{s5}

\subsection{Modulus of continuity for the ancestry process}

In this subsection we first obtain some estimates on the
$\Lambda$-coalescent and on the maximal dislocation of the particles
from their respective ancestors.

Denote by $\lfloor x\rfloor$ the integer part of $x$ for any $x\in\RR$.
Given $T>0$ and $\Delta>0$, we can divide the interval $[0,T]$ into
subintervals  as follows:
\begin{equation*}
[0,\Delta], [\Delta,2\Delta],\ldots,
\left[\lfloor T/\Delta-1\rfloor\Delta, \lfloor T/\Delta\rfloor\Delta\right],[\lfloor T/\Delta\rfloor\Delta, T].
\end{equation*}
Set $\Delta\equiv\Delta_n=2^{-n}$. Let $S_n^T$ be the collection of the endpoints of the first $\lfloor2^nT\rfloor$ subintervals, i.e.,
$$S_n^T\equiv\left\{k2^{-n}: 0\leq k\leq2^nT\right\}.$$ Put
\begin{equation*}
S^T\equiv\bigcup_{n\geq1}S_n^T=\bigcup_{n\geq 1}\left\{k2^{-n}: 0\leq k\leq2^nT\right\}.
\end{equation*} Clearly, given any $T>0$, $S^T$ is the collection of all the dyadic rationals in $[0,T]$. So $S^T$ is a dense subset of $[0, T]$.

For any $n\in[\infty]$, let $\left\{\mathbb{A}_{n,k}: 1\leq k\leq
2^nT\right\}$ be the collection of the first $\lfloor2^nT\rfloor$ subintervals in the partition so that
$$\mathbb{A}_{n,k}\equiv\left[{(k-1)}{2^{-n}},{k}{2^{-n}}\right].$$
For simplicity, we denote
$$N_{n,k}\equiv N^{\(k-1\)2^{-n}, k2^{-n}}.$$

Also denote by $H_{n,k}$ the maximal dislocation over interval
$\mathbb{A}_{n,k}$ of all the Brownian motions followed by the
countably many particles alive at time $k2^{-n}$ and their respective
ancestors at time $\(k-1\)2^{-n}$, i.e.,
\begin{equation*}
\begin{split}
H_{n,k}&\equiv H\(\(k-1\)2^{-n}, k2^{-n}\).
\end{split}
\end{equation*}
For any positive integer $m$, let
\begin{equation*}
T^{n,k}_m\equiv\inf\big\{t\in [0,2^{-n}]: \#\Pi^{k2^{-n}}(t)\leq m\big\}
\end{equation*}
with the convention $\inf\emptyset=2^{-n}$. Notice that for any
fixed $n\in[\infty]$ and $m$, the random times $\{T^{n,k}_m: 1\leq
k\leq2^nT\}$ follow the same distribution. Write
$T^{n,k}_x\equiv T^{n,k}_{\lfloor x\rfloor}$ for any $x>0$.

%The next lemma follows from the reflection principle and standard
%estimate on normal distribution. See Problem 9.22 of Karatzas and Shreve \cite{KS}.

We need a standard estimate on Brownian motion.

\begin{lemma}\label{0108}
Given any $x>0$ and $d$-dimensional standard Brownian motion
$\(\mathbf{B}\(s\)\)_{s\geq 0}$, we have
\begin{eqnarray*}
\mathbb{P}\(\sup_{\begin{smallmatrix}0\leq s\leq
t\end{smallmatrix}}|\mathbf{B}(s)|> x\)\leq
\sqrt{\frac{8d^3t}{\pi}}\frac{1}{x}\exp\(-\frac{x^2}{2dt}\).
\end{eqnarray*}
\end{lemma}

\begin{lemma}\label{121302}
Under Assumption I and for any $T>0$, there exists a positive constant
$C_4\(d,\al\)$ such that $\mathbb{P}$-a.s.
$$\max_{1\leq k\leq2^nT}H_{n,k}\leq C_4(d,\al)h\(2^{-n}\)$$
for $n$ large enough, where $h$ is defined by
(\ref{7112}).
\end{lemma}

\begin{proof}
Given any $n$ and $1\leq k\leq2^nT$, we first divide each interval $\mathbb{A}_{n,k}$ into countably many
subintervals as follows:
$$J_0^{n,k}\equiv\left[{(k-1)}{2^{-n}},{k}{2^{-n}}-T^{n,k}_{8^{(n+1)/\al}}\right]$$
and $$J_l^{n,k}\equiv\left[{k}{2^{-n}}-T^{n,k}_{8^{\(n+l\)/\al}},
{k}{2^{-n}}-T^{n,k}_{8^{\(n+l+1\)/\al}}\right]$$ for
$l=1,2,3,\ldots$.
 Consequently, the
lengths of these countably many subintervals satisfy that
$$\left|J_0^{n,k}\right|\leq{2^{-n}}\text{~~and~~}
\left|J_l^{n,k}\right|\leq
T^{n,k}_{8^{\(n+l\)/\al}}= T^{n,k}_{2^{\(3n+3l\)/\al}}\text{~~for~~}l=1,2,3,\ldots.$$

The right endpoints of these subintervals
$\(b^{n,k}_l\)_{l\geq1}\equiv\(k2^{-n}-T^{n,k}_{2^{\(3n+3l\)/\al}}\)_{l\geq
1}$ consist of a sequence of random times converging increasingly to
$k2^{-n}$. Set $b_0^{n,k}\equiv\(k-1\)2^{-n}$ for convenience.

For $l=0,1,2,\ldots$, let
$D_l^{n,k}$ be the maximal dislocation
of the ancestors (for those countably many particles alive at time $k2^{-n}$) at time $b_{l+1}^{n,k}$ from their respective ancestors at time $b_l^{n,k}$, i.e.,
\begin{equation}\label{8232}
\begin{split}
D_l^{n,k}\equiv&\max_{1\leq i\leq N^{b_l^{n,k},k2^{-n}}}\max_{j\in\pi_i}\left|X_{L_j^{k2^{-n}}\(b^{n,k}_{l+1}\)}\(b^{n,k}_{l+1}-\)-X_i\(b_l^{n,k}-\)\right|,
\end{split}
\end{equation}
where $\left\{\pi_i: 1\leq i\leq N^{b_l^{n,k},k2^{-n}}\right\}$
denotes the collection of all the disjoint blocks of partition
$\Pi^{k2^{-n}}\(k2^{-n}-b_l^{n,k}\)$ ordered by their least
elements.

In the case of $b_{l+1}^{n,k}=b_l^{n,k}$, i.e.,
$\left|J_l^{n,k}\right|=0$, which corresponds to the situation of either
$T^{n,k}_{2^{\(3n+3l+3\)/\al}}=2^{-n}$ or $T^{n,k}_{2^{\(3n+3l+3\)/\al}}=T^{n,k}_{2^{\(3n+3l\)/\al}}$,
it follows from Lemma
\ref{level} that
$$L_j^{k2^{-n}}\(b^{n,k}_{l+1}\)=L_j^{k2^{-n}}\(b^{n,k}_{l}\)=i$$
for any $j\in\pi_i$ with $1\leq i\leq N^{b_l^{n,k},k2^{-n}}$. Hence
we have $D_l^{n,k}=0$ in (\ref{8232}).

By the lookdown construction and the coming down from infinity
property, there exists a finite number of ancestors at each time
$b^{n,k}_l$, $l=0,1,2,\ldots$ for those countably many particles alive
at time $k2^{-n}$, i.e.,
$$\#\left\{L_j^{k2^{-n}}\(b_l^{n,k}\): j\in[\infty]\right\}<\infty.$$
So both maximums in (\ref{8232}) are in fact taken over finite sets.
Put
$$D^{n,k}\equiv\sum_{l=0}^{\infty}D_l^{n,k}.$$

For dimension $d$ and constant $\al$ in Assumption I, let
$C_1\(d,\al\)$ be a positive constant satisfying
$$C_1\(d,\al\)>\sqrt{2d\(3/\al+1\)}.$$

Now we estimate
%the dislocations over time interval
%$\left\{\mathbb{A}_{n,k}: 1\leq k\leq2^nT\right\}$ of
%all the Brownian motions carried out the countably many particles
%alive at time $k2^{-n}$ and their finite ancestors at time
%$\(k-1\)2^{-n}$ as follows.
the total maximal dislocation $D^{n,k}$ as follows. Let
\begin{equation*}
\begin{split}
I_n\equiv& \mathbb{P}\(\max_{1\leq k\leq
2^nT}D^{n,k}> \sum_{l=0}^{\infty}C_1\(d,\al\)h\(2^{-\(n+2l\)}\)\).\\
\end{split}
\end{equation*}

Since $D^{n,k}=\sum_{l=0}^{\infty}D_l^{n,k}$, we have
\begin{equation*}
\begin{split}
&\left\{D^{n,k}>
\sum_{l=0}^{\infty}C_1\(d,\al\)h\(2^{-\(n+2l\)}\)\right\}
\subseteq\bigcup_{l=0}^{\infty}\left\{D_l^{n,k}>
C_1\(d,\al\)h\(2^{-\(n+2l\)}\)\right\}.
\end{split}
\end{equation*}
Therefore,
\begin{equation*}
\begin{split}
I_n\leq&
\sum_{k=1}^{2^nT}\sum_{l=0}^{\infty}\mathbb{P}\(D_l^{n,k}>
C_1\(d,\al\)h\(2^{-\(n+2l\)}\)\).
\end{split}
\end{equation*}

Under Assumption I, there exists a positive constant $C$ such that for $\mathbf{N}$ large enough and for all $n>\mathbf{N}$, $\EE T_{8^{n/\al}}\leq C8^{-n}$.
For all those $n>\mathbf{N}$,
since $D_l^{n,k}=0$ for those $l$ with interval length $\left|J_l^{n,k}\right|=0$, we only need to
consider the case of $\left|J_l^{n,k}\right|>0$.

Observe that for $l=0,1,2,\ldots$, the total number of Brownian motion
paths connecting the ancestors (of the countably many particles
alive at $k2^{-n}$) at time $b^{n,k}_{l+1}$ to their respective
ancestors at earlier time $b^{n,k}_l$
 is at most $8^{\(n+l+1\)/\al}$. Since $|J_0^{n,k}|=b^{n,k}_1-b^{n,k}_0\leq 2^{-n}$, we
have
\begin{equation*}
\begin{split}
&\mathbb{P}\(D_0^{n,k}> C_1\(d,\al\)h\(2^{-n}\)\)
\leq8^{\frac{n+1}{\al}}\mathbb{P}\(\sup_{\begin{smallmatrix}0\leq
s\leq 2^{-n}\end{smallmatrix}}|\mathbf{B}\(s\)|>
C_1\(d,\al\)h\({2^{-n}}\)\).
\end{split}
\end{equation*}
For $l=1,2,\ldots,$ we have
\begin{equation*}
\begin{split}
&\mathbb{P}\(D_l^{n,k}> C_1\(d,\al\)h\(2^{-\(n+2l\)}\)\)\\
\leq&\mathbb{P}\(\left|J_l^{n,k}\right|>2^{-\(n+2l\)}\)+\mathbb{P}\(D_l^{n,k}>
C_1\(d,\al\)h\(2^{-\(n+2l\)}\),0<\left|J_l^{n,k}\right|\leq2^{-\(n+2l\)}\).
\end{split}
\end{equation*}
Since $|J_l^{n,k}|\leq T^{n,k}_{2^{(3n+3l)/\al}}$, for any
$n>\mathbf{N}$ the length of interval $J_l^{n,k}$ satisfies
\begin{equation*}
\begin{split}
\mathbb{P}\(\left|J_l^{n,k}\right|>2^{-\(n+2l\)}\)\leq&\mathbb{P}\(T^{n,k}_{2^{\(3n+3l\)/\al}}>2^{-\(n+2l\)}\)\\
\leq& 2^{n+2l}{\EE
T^{n,k}_{2^{\(3n+3l\)/\al}}}{}\leq {C}2^{-(2n+l)}.
\end{split}
\end{equation*}
 We further have
\begin{equation*}
\begin{split}
&\mathbb{P}\(D_l^{n,k}> C_1\(d,\al\)h\(2^{-\(n+2l\)}\)\)\\
\leq&{C}{2^{-\(2n+l\)}}+8^{\frac{n+l+1}{\al}}\mathbb{P}\(\sup_{\begin{smallmatrix}0\leq
s\leq 2^{-\(n+2l\)}\end{smallmatrix}}|\mathbf{B}\(s\)|>
C_1\(d,\al\)h\(2^{-\(n+2l\)}\)\).
\end{split}
\end{equation*}
Therefore,
\begin{equation*}
\begin{split}
I_{n}\leq& 2^nT8^{\frac{n+1}{\al}}\mathbb{P}\(\sup_{\begin{smallmatrix}0\leq
s\leq 2^{-n}\end{smallmatrix}}|\mathbf{B}\(s\)|>
C_1\(d,\al\)h\({2^{-n}}\)\)\\
&+2^nT\sum_{l=1}^{\infty}\({C}{2^{-\(2n+l\)}}+8^{\frac{n+l+1}{\al}}\mathbb{P}\(\sup_{\begin{smallmatrix}0\leq
s\leq 2^{-\(n+2l\)}\end{smallmatrix}}|\mathbf{B}\(s\)|>
C_1\(d,\al\)h\(2^{-\(n+2l\)}\)\)\)\\
=&\sum_{l=1}^{\infty}{CT}{2^{-\(n+l\)}}+2^nT\sum_{l=0}^{\infty}8^{\frac{n+l+1}{\al}}\mathbb{P}\(\sup_{\begin{smallmatrix}0\leq
s\leq 2^{-\(n+2l\)}\end{smallmatrix}}|\mathbf{B}\(s\)|>
C_1\(d,\al\)h\(2^{-\(n+2l\)}\)\).
\end{split}
\end{equation*}

It follows from Lemma \ref{0108} that
\begin{eqnarray*}
\begin{split}
&\mathbb{P}\(\sup_{\begin{smallmatrix}0\leq s\leq
2^{-\(n+2l\)}\end{smallmatrix}}|\mathbf{B}\(s\)|>
C_1\(d,\al\)h\(2^{-\(n+2l\)}\)\)\\
\leq&\frac{1}{C_1(d,\al)}\sqrt{\frac{8d^3}{\pi
(n+2l)\log2}}\exp\(-\frac{C^2_1(d,\al)\(n+2l\)\log2}{2d}\)\\
\leq&\frac{1}{C_1(d,\al)}\sqrt{\frac{8d^3}{\pi
\log2}}2^{-\frac{C_1^2(d,\al)\(n+2l\)}{2d}}\\
\equiv&C_2(d,\al)2^{-\frac{C_1^2(d,\al)\(n+2l\)}{2d}}.
\end{split}
\end{eqnarray*}
Therefore, for any $n>\mathbf{N}$ we have
\begin{equation*}
\begin{split}
I_{n}&\leq CT2^{-n}+2^nT\sum_{l=0}^{\infty}8^{\frac{n+l+1}{\al}}C_2(d,\al)2^{-\frac{C_1^2(d,\al)\(n+2l\)}{2d}}\\
&\leq{CT}{2^{-n}}+\sum_{l=0}^{\infty}TC_2(d,\al)2^{-\(\frac{C_1^2(d,\al)}{2d}-\frac{3}{\al}-1\)n-\(\frac{C_1^2(d,\al)}{d}-\frac{3}{\al}\)l+\frac{3}{\al}}.\\
\end{split}
\end{equation*}

Since $C_1\(d,\al\)>\sqrt{2d\(3/\al+1\)}$, it follows that
\begin{equation}\label{I_{n,1}}
\begin{split}
I_{n}\leq&{CT}{2^{-n}}+TC_3(d,\al)2^{-\(\frac{C_1^2(d,\al)}{2d}-\frac{3}{\al}-1\)n},
\end{split}
\end{equation}
where
$$C_3(d,\al)\equiv\sum_{l=0}^{\infty}C_2(d,\al)2^{-\(\frac{C_1^2(d,\al)}{d}-\frac{3}{\al}\)l+\frac{3}{\al}}.$$
Both terms on the right hand side of  (\ref{I_{n,1}}) are summable
with respect to $n$. Thus, $\sum_{n}I_{n}<\infty$, and it follows
from the Borel-Cantelli lemma that $\mathbb{P}$-a.s.
\begin{equation*}
\begin{split}
\max_{1\leq k\leq
2^nT}D^{n,k}
\leq&\sum_{l=0}^{\infty}C_1\(d,\al\)h\(2^{-\(n+2l\)}\)\\
\leq&C_1\(d,\al\)\sqrt{2^{-n}n\log2}\(1+\sum_{l=1}^{\infty}\sqrt{2^{-2l+1}l}\)\\
\equiv&C_4\(d,\al\)\sqrt{2^{-n}n\log2}
\end{split}
\end{equation*}
for $n$ large enough.

By the lookdown construction and the arguments in Lemmas 4.6-4.7
of \cite{LZ} we have $H_{n,k}\leq D^{n,k}$. Thus, $\mathbb{P}$-a.s.
\begin{equation*}
\max_{1\leq k\leq 2^nT}H_{n,k}\leq\max_{1\leq k\leq
2^nT}D^{n,k}\leq C_4(d,\al)h\(2^{-n}\)
\end{equation*}
for $n$ large enough.
\end{proof}

 Lemma \ref{ine} follows from the lookdown construction.
\begin{lemma}\label{ine}
For any $r,t,s$ with $0\leq r\leq t\leq s$ we have
\begin{equation*}
\begin{split}
&H\(r,s\)\leq H\(r,t\)+H\(t,s\)
%\text{~~and~~}H\(r,r+\)=H\(s-,s\)=0,
\end{split}
\end{equation*}
with the convention $H\(r,r\)=H\(s,s\)\equiv 0$.
\end{lemma}

We are ready to prove the one-sided modulus of continuity for the ancestry process.
\begin{proof}[{\bf Proof of Theorem \ref{th1}}]
We first show that $\mathbb{P}$-a.s. for all $r,s\in S^T$ satisfying $0<s-r\leq\theta,$
\begin{eqnarray*}
\begin{split}
H\(r,s\) \leq &Ch\(s-r\).
\end{split}
\end{eqnarray*}
The following argument is similar to that in Section III.1 of Perkins \cite{perkin}.

By Lemma \ref{121302}, given $T>0$, there exist an event
$\Omega_{T,d,\al}$ of probability one, and an integer-valued random
variable $ {\mathbf{N}}(T,d,\al)$ big enough such that
$2^{-{\mathbf{N}}(T,d,\al)}\leq e^{-1}$ and
\begin{eqnarray}\label{121303}
\max_{1\leq k\leq 2^nT}H_{n,k}\leq
C_4(d,\al)h\(2^{-n}\),\ \ n>{\mathbf{N}}(\omega,T,d,\al),\
\omega\in\Omega_{T,d,\al}.
\end{eqnarray}

Let $\theta\equiv\theta\(\omega,T,d,\al\)=2^{-{\mathbf{N}}(\omega,T,d,\al)}$.
For any $r,s\in S^T$ with $0<s-r\leq
2^{-{\mathbf{N}}(\omega,T,d,\al)}=\theta$, there exists an
$n\geq{\mathbf{N}}(\omega,T,d,\al)$ such that $2^{-\(n+1\)}<s-r\leq
2^{-n}$. Recall that
$$S^T_k=\left\{l2^{-k}: 0\leq l\leq2^kT\right\}\text{~~and~~}\overline{S^T}=\overline{\cup_{k\geq 1}S^T_k}=[0,T].$$ For any $k> n$, choose $s_k\in S^T_k$ such that $s_k\leq s$ and $s_k$ is
the largest such value. Then
$$s_k\uparrow s,\ \
s_{k+1}=s_k+j_{k+1}2^{-\(k+1\)}\text{~~with~~}j_{k+1}\in\{0,1\}.$$
Since $s\in S^T$, then $\(s_k\)_{k>n}$ is a sequence with at most finite terms that are not equal to $s$.
Applying  (\ref{121303}), we have
\begin{equation}\label{130805}
H\(s_k,s_{k+1}\)\leq C_4(d,\al) j_{k+1}h\(2^{-\(k+1\)}\).
\end{equation}

By Lemma \ref{ine},
\begin{eqnarray}\label{12415}
\begin{split}
H\(s_{n+1},s\)
%\leq&H\(s_{n+1},s-\)+H\(s-,s\)
\leq&\sum_{k=n+1}^{\infty}H\(s_k,s_{k+1}\)\\
\leq&\sum_{k=n+1}^{\infty}C_4(d,\al)j_{k+1}h\(2^{-\(k+1\)}\)\\
\leq&C_4(d,\al)\sum_{k=n+1}^{\infty}\sqrt{2^{-(k+1)}\(k+1\)\log2}\\
%=&C_4(d,\al)\sqrt{2^{-(n+1)}\(n+1\)\log2}\sum_{k=n+1}^{\infty}\sqrt{2^{-(k-n)}\frac{k+1}{n+1}}\\
%=&C_4(d,\al)\sqrt{2^{-(n+1)}\(n+1\)\log2}\sum_{k=1}^{\infty}\sqrt{2^{-{k}}\(1+\frac{k}{n+1}\)}\\
\leq&C_4(d,\al)\sqrt{2^{-(n+1)}\(n+1\)\log2}\sum_{k=1}^{\infty}\sqrt{2^{-{k+1}}k}\\
\equiv & C_5(d,\al)\sqrt{2^{-(n+1)}\(n+1\)\log2},
\end{split}
\end{eqnarray}
where observe that only finitely many terms are nonzero in the
summation on the right hand side of the first inequality.

Similarly, for any $k> n$, choose $r_k\in S^T_k$ such that $r_k\geq r$ and $r_k$ is
the smallest such value. Then
$$r_k\downarrow r,\ \
r_{k+1}=r_k-j^{'}_{k+1}2^{-\(k+1\)}\text{~~with~~}j{'}_{k+1}\in\{0,1\}.$$
Applying  (\ref{121303}), we have
\begin{eqnarray*}
H\(r_{k+1},r_{k}\)\leq C_4(d,\al)j^{'}_{k+1}h\(2^{-\(k+1\)}\).
\end{eqnarray*}
Similar to (\ref{12415}), by Lemma \ref{ine} we have
\begin{eqnarray}\label{p2}
\begin{split}
H\(r,r_{n+1}\)
\leq&\sum_{k=n+1}^{\infty}H\(r_{k+1},r_k\)\\
\leq&\sum_{k=n+1}^{\infty}C_4(d,\al)j^{'}_{k+1}h\(2^{-\(k+1\)}\)\\
\leq& C_5(d,\al)\sqrt{2^{-(n+1)}\(n+1\)\log2}.\\
\end{split}
\end{eqnarray}

Since $2^{-\(n+1\)}< s-r\leq 2^{-n}$, we have $0\leq
s_{n+1}-r_{n+1}\leq i_{n+1}2^{-\(n+1\)}$ with $i_{n+1}\in\{0,1,2\}$.
It comes from (\ref{130805}) and Lemma \ref{ine} that
\begin{eqnarray}\label{p3}
\begin{split}
H\(r_{n+1},s_{n+1}\)\leq&
2C_4(d,\al)h\(2^{-\(n+1\)}\)\\
=&2C_4(d,\al)\sqrt{2^{-(n+1)}\(n+1\)\log2}.
\end{split}
\end{eqnarray}

Combining (\ref{12415}), (\ref{p2}) and (\ref{p3}), we have $\mathbb{P}$-a.s. for all $r,s\in S^T$ with $0<s-r\leq\theta$
\begin{eqnarray*}
\begin{split}
H\(r,s\)
\leq&
%\max_{\begin{smallmatrix}r,s\in S^T\\0<s-r\leq\theta\end{smallmatrix}}
H\(r,r_{n+1}\)+H\(r_{n+1},s_{n+1}\)+H\(s_{n+1},s\)\\
\leq&2C_4(d,\al)\sqrt{2^{-(n+1)}\(n+1\)\log2}+2C_5(d,\al)\sqrt{2^{-(n+1)}\(n+1\)\log2}\\
\leq &C(d,\al)\sqrt{2^{-(n+1)}\(n+1\)\log2},\\
\end{split}
\end{eqnarray*}
where $C(d,\al)\equiv2C_4(d,\al)+2C_5(d,\al)$.

Function $h$ is increasing on $(0,e^{-1}]$. Since
$$2^{-\(n+1\)}<s-r\leq\theta\leq e^{-1},$$ we have
\begin{equation}\label{919e1}
\begin{split}
%\max_{\begin{smallmatrix}r,s\in S^T\\0<s-r\leq\theta\end{smallmatrix}}
H\(r,s\)
\leq&C(d,\al)h\(2^{-(n+1)}\) \leq C(d,\al)h\(s-r\)
\end{split}
\end{equation}
for all $r,s\in S^T$ satisfying $0<s-r\leq\theta$.

Finally, for any $0<r<s<T$ with $s-r<\theta/2$, find sequences $(r_m)\subseteq S^T$ and $(s_n)\subseteq S^T$ with $r_m\uparrow r$ and $s_n\downarrow s$.
By the lookdown construction, for any $j\in [\infty]$,
\begin{equation}\label{914e1}
\begin{split}
&|X_j(s)-X_{L^s_j(r)}(r-)|\\
\leq&|X_j(s)-X_j(s_n)|+|X_j(s_n)-X_{L^{s_n}_j(r_m)}(r_m-)|\\
&+|X_{L^{s_n}_j(r_m)}(r_m-)-X_{L_j^{s_n}(r)}(r-)|+|X_{L_j^{s_n}(r)}(r-)-X_{L_j^s(r)}(r-)|.
\end{split}
\end{equation}
Let both $n$ and $m$ be big enough such that $0<s_n-r_m\leq \theta$.
It follows from (\ref{919e1}) that the second term on the right hand
side of (\ref{914e1}) is bounded from above by
$C\(d,\al\)h\(s_n-r_m\)$. First fix $n$ and let
$m\rightarrow\infty$. The third term tends to $0$ because
$X_{L_j^{s_n}(\cdot)}(\cdot-)$ is continuous for any $j\in[\infty]$. Then letting
$n\rightarrow\infty$, the first term tends to $0$ because $X_j(\cdot)$ is right continuous for any $j\in[\infty]$. The last term is
equal to $0$ for large $n$ since  $s_n$ is then so close  to
$s$ that there is no lookdown event involving levels
$\left\{1,2,\ldots,j\right\}$ during time interval $(s,s_n]$.
Consequently,
\begin{equation*}
\begin{split}
&|X_j(s)-X_{L^s_j(r)}(r-)|\\
\leq&\lim_{n\rightarrow\infty}|X_j(s)-X_j(s_n)|+\lim_{n\rightarrow\infty}\lim_{m\rightarrow\infty}C\(d,\al\)h\(s_n-r_m\)\\
&+\lim_{n\rightarrow\infty}\lim_{m\rightarrow\infty}|X_{L^{s_n}_j(r_m)}(r_m-)-X_{L_j^{s_n}(r)}(r-)|
+\lim_{n\rightarrow\infty}|X_{L_j^{s_n}(r)}(r-)-X_{L_j^s(r)}(r-)|\\
=&C\(d,\al\)h\(s-r\).
\end{split}
\end{equation*}
Then (\ref{mod1}) follows.
\end{proof}

\begin{remark}
It follows from estimate (\ref{I_{n,1}}) that there exist positive
constants $C_6\equiv C_6(T, d,\al)$ and $C_7\equiv C_7(d,\al)$ such
that for $\ep>0$ small enough
\[ \mathbb{P} (\theta \leq\ep)\leq C_6 {\ep}^{C_7}.\]
\end{remark}

\subsection{Modulus of continuity for the $\Lambda$-Fleming-Viot support process and uniform compactness for the support and range}

We will need the following observation on weak convergence.
\begin{lemma}\label{901t1}
If $\left\{\(\nu_n\)_{n\geq 1},\nu\right\}\subseteq M_1\(\RR^d\)$ and $\nu_n$ weakly converges to $\nu$, then we have
$$\text{supp}\,\nu\subseteq\cap_{m\geq1}\overline{\cup_{n\geq m}\text{supp}\,\nu_n}.$$
\end{lemma}

\begin{proof}
Suppose that there exists an $x\in \bR^d$ such that
$$x\in\text{supp}\,\nu\cap{\overline{\cup_{n\geq m}\text{supp}\,\nu_n}}^{\text{\ c}}$$
for some $m$. Since ${\overline{\cup_{n\geq m}\text{supp}\,\nu_n}}^{\text{\ c}}$ is an open set, there exists a positive value $\delta$ such that $\left\{y:|y-x|<\delta\right\}\subseteq{\overline{\cup_{n\geq m}\text{supp}\,\nu_n}}^{\text{\ c}}$. We can define a nonnegative and continuous function $g$ satisfying $g>0$ on $\left\{y:|y-x|<\delta/2\right\}$ and $g=0$ on $\left\{y:|y-x|\geq\delta\right\}$. Then $\left<\nu_{n},g\right>=0$ for any $n\geq m$ but $\left<\nu,g\right>>0$. Consequently, $\left<\nu_{n},g\right>\not\rightarrow\left<\nu,g\right>$, which contradicts the fact that $\nu_n$ weakly converges to $\nu$.
\end{proof}

\begin{proof}[{\bf Proof of Theorem \ref{modulus2}}]
Applying Theorem \ref{th1}, there exist a positive random variable
$\theta\equiv\theta\(T,d,\al\)$  and a constant $C\equiv C\(d,\al\)$
such that given any fixed $t\in[0,T)$, $\mathbb{P}$-a.s. for all $r\in
S^{T}\cap\left(t,t+\theta\right]$, we have  $$H\(t,r\)\leq
Ch\(r-t\),$$
which gives the upper bound for the maximal dislocation between the countably many particles at time $r$ and their corresponding ancestors at time $t$.
By Lemma \ref{level}, the ancestors at time $t$
are exactly $\left\{X_1\(t-\),X_2\(t-\),\ldots,X_{N^{t,r}}\(t-\)\right\}$, so we have $\mathbb{P}$ a.s.
$$\left\{X_1\(r\),X_2\(r\),\ldots\right\}\subseteq\bigcup_{1\leq i\leq N^{t,r}}\mathbb{B}\(X_i\(t-\),Ch\(r-t\)\).$$
For the given $t\in[0,T)$, $\mathbb{P}$ a.s.
$$X_i(t)=X_i(t-)\text{~~for~any~}i\in[\infty],$$
where $X_i(0-)\equiv X_i(0)$,
so for any $r\in
S^{T}\cap\left(t,t+\theta\right]$, we have $\mathbb{P}$ a.s.
\begin{equation}\label{201319}
\begin{split}
\left\{X_1\(r\),X_2\(r\),\ldots\right\}
\subseteq\bigcup_{1\leq i\leq N^{t,r}}\mathbb{B}\(X_i\(t\),Ch\(r-t\)\).
\end{split}
\end{equation}
Apply Lemma \ref{support}, for the given $t\in[0,T)$, $\mathbb{P}$ a.s.
$$\left\{X_1\(t\),X_2\(t\),\ldots,X_{N^{t,r}}\(t\)\right\}\subseteq\text{supp}\,X\(t\).$$
It follows from (\ref{201319}) that
$$\left\{X_1\(r\),X_2\(r\),\ldots\right\}\subseteq\mathbb{B}\(\text{supp}\,X\(t\),Ch\(r-t\)\).$$
For all $r\in
S^T\cap(t,t+\theta]$, we have $\mathbb{P}$-a.s.
$$X^{\(n\)}\(r\)\equiv\frac{1}{n}\sum_{i=1}^{n}\delta_{X_i\(r\)}\rightarrow X\(r\).$$
Clearly,
$$\text{supp}\,X^{\(n\)}\(r\)\subseteq\left\{X_1\(r\),X_2\(r\),\ldots\right\}\subseteq\mathbb{B}\(\text{supp}\,X\(t\),Ch\(r-t\)\)$$
for all $n$, which implies
\begin{equation}\label{916e1}
\text{supp}\,X\(r\)\subseteq\mathbb{B}\(\text{supp}\,X\(t\),Ch\(r-t\)\).
\end{equation}

Then for any $s$ satisfying $t<s\leq \(t+\theta/2\)\wedge T$, we can choose a
sequence $\(s_l\)_{l\geq1}\subseteq S^T\cap\left (t,t+\theta\right]$ such that
$s_l\downarrow s$. It follows from  the right continuity of $X$ and Lemma \ref{901t1}  that
$$\text{supp}\,X\(s\)\subseteq\bigcap_{m\geq1}\overline{\bigcup_{l\geq m}\text{supp}\,X\(s_l\)}.$$ By (\ref{916e1}), we have
$$\text{supp}\,X\(s_l\)\subseteq\mathbb{B}\(\text{supp}\,X(t), Ch\(s_l-t\)\)$$
for all $l$. Consequently, for any $t<s\leq \(t+\theta/2\)\wedge T$,
\begin{equation*}
\begin{split}
\text{supp}\,X\(s\)&\subseteq\bigcap_{m\geq1}\overline{\bigcup_{l\geq m}\mathbb{B}\(\text{supp}\,X(t), Ch\(s_l-t\)\)}\\
&=\bigcap_{m\geq1}{\mathbb{B}\(\text{supp}\,X(t), Ch\(s_m-t\)\)}\\
&=\mathbb{B}\(\text{supp}\,X(t), Ch\(s-t\)\).
\end{split}
\end{equation*}

Therefore, given any fixed $t\geq0$, there exist a
positive random variable $\theta\equiv\theta\(t,d,\al\)$ and a
constant $C\equiv C(d,\al)$ such that for any $\Delta t$ with
$0<\Delta t\leq\theta$, $\mathbb{P}$-a.s.
\begin{eqnarray*}
\begin{split}
\text{supp}\,X\(t+\Delta t\)\subseteq\mathbb{B}\(\text{supp}\,X(t), Ch\(\Delta t\)\)=\mathbb{B}\(\text{supp}\,X(t), C\sqrt{\Delta t\log\(1/\Delta t\)}\).
\end{split}
\end{eqnarray*}
\end{proof}

\begin{remark}
The constants $C\equiv C\(d,\al\)$ in Theorems \ref{th1} and \ref{modulus2} are the same. From the proofs of Lemma \ref{121302} and Theorems \ref{th1}-\ref{modulus2}, it is clear that
\begin{equation*}
\begin{split}
C\(d,\al\)&=2C_4(d,\al)+2C_5(d,\al)\\
&=2C_4(d,\al)+2C_4(d,\al)\sum_{k=1}^{\infty}\sqrt{2^{-k+1}k}\\
&=2C_1(d,\al)\(1+\sum_{l=1}^{\infty}\sqrt{2^{-2l+1}l}\)\(1+\sum_{k=1}^{\infty}\sqrt{2^{-k+1}k}\),
\end{split}
\end{equation*}
where $C_1(d,\al)$ is any constant satisfying $C_1(d,\al)>\sqrt{2d\(3/\al+1\)}$.
\end{remark}

\begin{lemma}\label{N_{n,k}}
Under Assumption I, we have $\mathbb{P}$-a.s.
\begin{equation*}
\max_{1\leq
k\leq 2^nT}N_{n,k}<4^{\frac{n}{\al}}n^{\frac{2}{\al}}
\end{equation*}
for $n$ large enough.
\end{lemma}
\begin{proof}
Under Assumption I, there exists a positive constant $C$ such that
\begin{equation}\label{VIE1}
\EE T_m\leq Cm^{-\al}
\end{equation}
for $m$ large enough.

Given $n$,   $T^{n,k}_{4^{n/\al}n^{2/\al}}, 1\leq k\leq 2^nT$ are
i.i.d. random variables following the same distribution as
$T_{4^{n/\al}n^{2/\al}}\wedge 2^{-n}$. Consequently, $N_{n,k}, 1\leq
k\leq 2^nT$ are also i.i.d. random variables. Choosing
$4^{n/\al}n^{2/\al}$ large enough, by (\ref{VIE1}) we have
\begin{equation*}
\begin{split}
\mathbb{P}\(\max_{1\leq
k\leq2^nT}N_{n,k}\geq4^{\frac{n}{\al}}n^{\frac{2}{\al}}\)
&=1-\prod_{1\leq k\leq2^nT}\(1-\mathbb{P}(N_{n,k}\geq4^{\frac{n}{\al}}n^{\frac{2}{\al}})\) \\
&\leq 2^nT\mathbb{P}\(N_{n,1}\geq4^{\frac{n}{\al}}n^{\frac{2}{\al}}\)\\
&=2^nT\mathbb{P}\(T_{4^{n/\al}n^{2/\al}}\geq 2^{-n}\)\\
&\leq2^nT\EE T^{n,k}_{4^{n/\al}n^{2/\al}}/2^{-n}\\
&\leq CT{n^{-2}},
\end{split}
\end{equation*}
which is summable with respect to $n$. Applying Borel-Cantelli
lemma, we then have $\mathbb{P}$-a.s.
\begin{equation*}
\max_{1\leq k\leq2^nT}N_{n,k}<4^{\frac{n}{\al}}n^{\frac{2}{\al}}
\end{equation*}
for $n$ large enough.
\end{proof}

\begin{proof}[{\bf Proof of Theorem \ref{916co1}}]
Under Assumption I, by Lemma \ref{N_{n,k}} we have $\mathbb{P}$-a.s.
\begin{equation}\label{121018}
\begin{split}
\max_{1\leq
k\leq 2^nT}N_{n,k}<4^{\frac{n}{\al}}n^{\frac{2}{\al}}
\end{split}
\end{equation}
for $n$ large enough.

Given any positive constants $\sigma$ and $T$ with $0<\sigma<T$, we first show that
$\mathcal{R}([\sigma,T))$ is a.s. compact.
Applying Theorem \ref{th1}, there exist a positive random variable
$\theta\equiv\theta\(T,d,\al\)>0$ and a constant $C\equiv C(d,\al)$
such that $\mathbb{P}$-a.s. for all $r,s\in S^{T}$ satisfying $0<s-r\leq\theta,$
\begin{eqnarray*}
\begin{split}
H\(r,s\) \leq &Ch\(s-r\).
\end{split}
\end{eqnarray*}
For the given $\sigma$, choose $n$ big enough so that
$2^{-n}\leq\theta\wedge\sigma$ and (\ref{121018}) holds. For any
$1\leq k\leq 2^nT$ and $t\in S^T\cap[k2^{-n},(k+1)2^{-n}\wedge T)$,
we have
\begin{equation*}
\begin{split}
H\(\(k-1\)2^{-n},t\)&\leq H\(\(k-1\)2^{-n},k2^{-n}\)+H\(k2^{-n},t\)\\
&\leq 2Ch\(2^{-n}\).
\end{split}
\end{equation*}
It follows from the lookdown construction and Lemma \ref{level} that
\begin{equation*}
\text{supp}\,X\(t\)\subseteq\bigcup_{1\leq i\leq N^{\(k-1\)2^{-n},t}}\mathbb B\(X_{i}\(\(k-1\)2^{-n}-\),2Ch\(2^{-n}\)\).
\end{equation*}
By (\ref{121018}) we have
$$N^{\(k-1\)2^{-n},t}\leq N^{\(k-1\)2^{-n},k2^{-n}}=N_{n,k}<4^{n/\al}n^{2/\al}.$$
Consequently,
\begin{equation}\label{121018e2}
\text{supp}\,X\(t\)\,\subseteq\bigcup_{1\leq i< 4^{n/\al}n^{2/\al}}\mathbb B\(X_{i}\(\(k-1\)2^{-n}-\),2Ch\(2^{-n}\)\).
\end{equation}
 For general $t\in[k2^{-n},(k+1)2^{-n}\wedge T)$. We can select a decreasing sequence $$\(t^{n,k}_l\)_{l\geq1}\subseteq S^T\cap[k2^{-n},(k+1)2^{-n}\wedge T)\text{~~satisfying~~}t^{n,k}_l\downarrow t\text{~as~}l\rightarrow\infty.$$
Since the $\La$-Fleming-Viot process $X$ is right continuous, it follows from
Lemma \ref{901t1} that
\begin{equation*}
\begin{split}
\text{supp}\,X\(t\)\subseteq &\bigcap_{m\geq1}\overline{\bigcup_{l\geq m}\text{supp}\,X\(t^{n,k}_l\)}.\\
\end{split}
\end{equation*}
By (\ref{121018e2}), we have
\begin{equation*}
\text{supp}\,X\(t^{n,k}_l\)\,\subseteq\bigcup_{1\leq i< 4^{n/\al}n^{2/\al}}\mathbb B\(X_{i}\(\(k-1\)2^{-n}-\),2Ch\(2^{-n}\)\).
\end{equation*}
Therefore, for any $t\in[k2^{-n},(k+1)2^{-n}\wedge T)$, we also have
\begin{equation}\label{121018e3}
\text{supp}\,X\(t\)\subseteq\bigcup_{1\leq i<4^{n/\al}n^{2/\al}}\mathbb B\(X_{i}\(\(k-1\)2^{-n}-\),2Ch\(2^{-n}\)\),
\end{equation}
i.e., $\mathcal{R}\([k2^{-n},(k+1)2^{-n}\wedge T)\)$ is contained in at most $\lfloor4^{n/\al}n^{2/\al}\rfloor$ closed balls each of which has radius bounded from above by $2Ch\(2^{-n}\)$.
Then
\begin{equation}\label{121018e4}
\begin{split}
\mathcal{R}\(\left[\sigma,T\right)\)&\subseteq\mathcal{R}\(\left[2^{-n},T\right)\)\\
&\subseteq\bigcup_{1\leq k\leq2^nT}\mathcal{R}\(\left[k2^{-n},\(k+1\)2^{-n}\wedge T\right)\)\\
&\subseteq\bigcup_{1\leq k\leq2^nT}\bigcup_{1\leq i<4^{n/\al}n^{2/\al}}\mathbb B\(X_{i}\(\(k-1\)2^{-n}-\),2Ch\(2^{-n}\)\),
\end{split}
\end{equation}
where the right hand side is the union of at most $\lfloor 2^nT\rfloor\times \lfloor4^{n/\al}n^{2/\al}\rfloor$ closed and bounded balls. So ${\mathcal{R}\(\left[\sigma,T\right)\)}$ is compact.

Consequently, the random measure $X(t)$ has compact support for all times $t\in[\sigma, T)$ simultaneously.
Let $\sigma=1/T$ and $T\rightarrow\infty$. Then the random measure $X\(t\)$ has compact support
for all times $t\in (0,\infty)$ simultaneously.

Further, given that supp\,$X(0)$ is compact, we can adapt the
above-mentioned strategy to find a finite cover for
$\mathcal{R}([0,T))$. Applying Theorem \ref{modulus2}, for $n$ large
enough, we have
\begin{equation*}
\begin{split}
\mathcal{R}\([0,2^{-n})\)=\overline{\bigcup_{t\in[0,2^{-n})}\text{supp}\,X(t)}\subseteq\mathbb{B}\(\text{supp}\,X\(0\), Ch\(2^{-n}\)\).
\end{split}
\end{equation*}
Then
\begin{equation*}
\begin{split}
&\mathcal{R}\([0,T)\)\\
&\subseteq\bigcup_{0\leq k\leq2^nT}\mathcal{R}\(\left[k2^{-n},\(k+1\)2^{-n}\wedge T\right)\)\\
&\subseteq\mathbb{B}\(\text{supp}\,X\(0\), Ch\(2^{-n}\)\)\bigcup\(\bigcup_{1\leq k\leq2^nT}\bigcup_{1\leq i<4^{n/\al}n^{2/\al}}\mathbb B\(X_{i}\(\(k-1\)2^{-n}-\),2Ch\(2^{-n}\)\)\),
\end{split}
\end{equation*}
where the right hand side is compact given the compactness of
supp\,$X\(0\)$. So, $\mathcal{R}\([0,T)\)$ is compact.

Note that $\mathcal{R}\([0,T)\)$ is increasing with respect to $T$. Let $T\rightarrow\infty$. It is clear that $\mathcal{R}\([0,t)\)$ is compact for all $t>0$ $\mathbb{P}$-a.s..
\end{proof}

%%%%%%%%%%%%%%%%%%%%%%%%%%%%%%%%%%%%%%%%%%%%%%%%%%%%%%%%%%%%%%%%%%%%%%%%%%%%%%%%%%%%%%%%%%%%%%%%%%%%%%%%%%%%%%%%%%%%
%%%%%%%%%%%%%%%%%%%%%%%%%%%%%%%%%%%%%%%%%%%%%%%%%%%%%%%%%%%%%%%%%%%%%%%%%%%%%%%%%%%%%%%%%%%%%%%%%%%%%%%%%%%%%%%%%%%%%%%%%
\subsection{Upper bounds on Hausdorff dimensions for the support and range}
Given any $\La$-coalescent $\(\Pi(t)\)_{t\geq0}$ with
$\Pi(0)=\mathbf{0}_{[\infty]}$, recall that
\begin{equation*}
T_m\equiv \inf\big\{t\geq0: \#\Pi(t)\leq m\big\}
\end{equation*}
with the convention $\inf\emptyset=\infty$. $\(\Pi_n(t)\)_{t\geq0}$ is its restriction to $\left[n\right]$ with $\Pi_n(0)=\mathbf{0}_{[n]}$.  For any $n\geq m$, let
$$T^{n}_m\equiv\inf\Big\{t\geq0: \#\Pi_n\(t\)\leq m\Big\}$$ with the convention $\inf\emptyset=\infty$.

For any $x>0$, write $T^{n}_x\equiv T^{n}_{\lfloor x\rfloor}\text{~~and~~}T_x\equiv T_{\lfloor x\rfloor}$.

Let $(\hat{T}_n)_{n\geq 2}$ be independent random variables such
that $\hat{T}_n$ has the same distribution as $T^{n}_{n-1}$.

\begin{lemma}\label{7151}
For any $n>m$, $T^{n}_m$ is stochastically less than
$\sum_{i=m+1}^n \hat{T_i}$, i.e., for any $t>0$,
\begin{equation}\label{sto_order}
\mathbb{P}\(T^{n}_m\geq t \)\leq \mathbb{P}\(\sum_{i=m+1}^n \hat{T_i}\geq t\).
\end{equation}
\end{lemma}
\begin{proof}
We use a coupling argument by defining an auxiliary
$[n]\times[n]$-valued continuous time Markov chain $(Y_1, Y_2)$
describing the following urn model. Intuitively, there are balls in
an urn of color either white or black. Let $Y_1(t)$ and $Y_2(t)$
represent the number of white and black balls at time $t$,
respectively.

After each independent exponential sampling time a random number of
balls are taken out of the urn and then immediately replaced with
certain white or black colored balls so that the total number of
balls in the urn decreases exactly by one overall afterwards. More
precisely, given that there are $w$ white balls and $b$ black balls
in the urn, at rate $\lambda_{w+b,k}$ each group of $k$ balls with
$k\leq w+b$ is independently removed. Suppose that $w'$ white balls
and $k-w'$ black balls have been chosen and removed at time $t$, we
then immediately return $k-1$ balls to the urn so that among the
returned balls, either one is white  and all the others are black if
$w'>0$ or all of them are black if $w'=0$. At such a sampling time
$t$ we define
\begin{equation*}
\begin{cases}
Y_1(t)=w-w'+1 \text{\,\, and \,\,}
Y_2(t)=b+w'-2=w+b-1-Y_1(t),\text{~~if~~}w'>0;\\
Y_1(t)=w\text{~~and~~}Y_2(t)=b-1,\text{~~if~~}w'=0,
\end{cases}
\end{equation*}
and the value of $(Y_1,Y_2)$ keeps unchanged between the sampling
times. The above-mentioned procedure continues until there is one
white ball left in the urn. Suppose that there are $n$ white balls
and no black balls in the urn initially, i.e.,
$(Y_1(0),Y_2(0))=(n,0)$.

Observe that $Y_1$ follows the law of the $\Lambda$-coalescent
starting with $n$-blocks and $(\hat{T_i})_{i\leq n}$ has the same
distribution as the inter-decreasing times for process $Y_1+Y_2$.
Plainly,
\[\inf\{t: Y_1(t)\leq m\}\leq \inf\{t: Y_1(t)+Y_2(t)\leq m\}.\]
Inequality (\ref{sto_order}) thus follows.
\end{proof}

The estimate in Lemma \ref{N_{n,k}} is not enough for the proofs of Theorems \ref{811t1}-\ref{range}. A sharper estimate is obtained in the following result under a stronger condition.

\begin{lemma}\label{N_{n,k,2}}
Suppose that Condition A holds. We have $\mathbb{P}$-a.s.
\begin{equation}\label{824e1}
\max_{1\leq
k\leq 2^nT}N_{n,k}<2^{\frac{n}{\al}}n^{\frac{2}{\al}}
\end{equation}
for $n$ large enough.
\end{lemma}

\begin{proof}
Under Condition A, there exists a positive constant $C$ such that for $n$ large enough and for any $b>2^{{n}/{\al}}n^{{2}/{\al}}$,
\begin{equation}\label{811e1}
\la_{b}\geq(C\lfloor2^{{n}/{\al}}n^{{2}/{\al}}\rfloor^{-\alpha})^{-1}>2^{n+1}n.
\end{equation}
Letting $n\rightarrow\infty$ in (\ref{sto_order}), for
any $t> 0$ and $m\in[\infty]$ we have
\begin{equation}\label{802VI2}
\mathbb{P}\(T_m\geq t \)\leq \mathbb{P}\(\sum_{i>m}\hat{T_i}\geq
t\).
\end{equation}

With  estimate (\ref{802VI2}) we can find a sharper uniform upper
bound for the maximal number of ancestors as follows:
\begin{equation*}
\begin{split}
\mathbb{P}\(\max_{1\leq
k\leq 2^nT}N_{n,k}\geq2^{\frac{n}{\al}}n^{\frac{2}{\al}}\)
%\leq&\mathbb{P}\(\max_{1\leq
%k\leq 2^nT}T^{n,k}_{2^{{n}/{\al}}n^{{2}/{\al}}}\geq2^{-n}\)\\
%\mathbb{P}\(\max_{1\leq
%k\leq2^nT}N_{n,k}\geq4^{\frac{n}{\al}}n^{\frac{2}{\al}}\)
&=1-\prod_{1\leq k\leq2^nT}\(1-\mathbb{P}(N_{n,k}\geq2^{\frac{n}{\al}}n^{\frac{2}{\al}})\) \\
&\leq 2^nT\mathbb{P}\(N_{n,1}\geq2^{\frac{n}{\al}}n^{\frac{2}{\al}}\)\\
&\leq2^nT \mathbb{P}\(T_{2^{{n}/{\al}}n^{{2}/{\al}}}\geq2^{-n}\)\\
&\leq2^nT\mathbb{P}\(\sum_{i>{2^{{n}/{\al}}n^{{2}/{\al}}}}\hat{T_i}\geq2^{-n}\)\\
%=&2^nT\mathbb{P}\(\exp{\(2^nn\(\sum_{i=\left\lfloor{2^{{n}/{\al}}n^{{2}/{\al}}}\right\rfloor+1}^{\infty} \hat{T_i}\)\)}\geq e^n\)\\
&\leq{2^nT}{e^{-n}}\EE\exp\(\sum_{i>2^{{n}/{\al}}n^{{2}/{\al}}} 2^nn\hat{T_i}\)\\
&={2^nT}{e^{-n}}\prod_{i>2^{{n}/{\al}}n^{{2}/{\al}}} \EE\exp\(2^nn\hat{T_i}\),
\end{split}
\end{equation*}
where $\hat{T}_i$ follows an exponential distribution with parameter
$\la_i$. It follows from (\ref{811e1}) that when $n$ is large
enough, $\la_{i}>2^nn$ for any
$i>{2^{{n}/{\al}}n^{{2}/{\al}}}$,  which
guarantees the existence of moment generating function for
$\hat{T}_i$. As a result,

\begin{equation*}
\begin{split}
\mathbb{P}\(\max_{1\leq
k\leq2^nT}N_{n,k}\geq2^{\frac{n}{\al}}n^{\frac{2}{\al}}\)
\leq&
{2^nT}{e^{-n}}\prod_{i>2^{{n}/{\al}}n^{{2}/{\al}}}\frac{\la_i}{\la_i-n2^n}\\
\equiv&{2^nT}{e^{-n}}Q.
\end{split}
\end{equation*}

Then
\begin{equation*}
\begin{split}
\ln Q=&\sum_{i>2^{{n}/{\al}}n^{{2}/{\al}}}\ln\(1+\frac{n2^n}{\la_i-n2^n}\)\\
\leq&\sum_{i>2^{{n}/{\al}}n^{{2}/{\al}}}\frac{n2^n}{\la_i-n2^n}\\
\leq&n2^n\sum_{i>2^{{n}/{\al}}n^{{2}/{\al}}}\frac{1}{\la_i-\la_i/2}\\
\leq&n2^{n+1}\sum_{i>2^{{n}/{\al}}n^{{2}/{\al}}}\frac{1}{\la_i}.
\end{split}
\end{equation*}
We have by Condition A for $n$ large enough,
\begin{equation*}
\begin{split}
\ln Q
\leq&n2^{n+1}C\(\lfloor2^{\frac{n}{\al}}n^{\frac{2}{\al}}\rfloor\)^{-\al}
\leq n2^{n+1}C\(2^{\frac{n}{\al}}n^{\frac{2}{\al}}/2\)^{-\al}
=2^{\al+1}Cn^{-1}.
\end{split}
\end{equation*}
Then
\[\sum_{n}\mathbb{P}\(\max_{1\leq
k\leq2^nT}N_{n,k}\geq2^{\frac{n}{\al}}n^{\frac{2}{\al}}\)<\infty,\]
which, by the Borel-Cantelli lemma, implies that $\mathbb{P}$-a.s.
\begin{equation*}
\max_{1\leq
k\leq2^nT}N_{n,k}<2^{\frac{n}{\al}}n^{\frac{2}{\al}}
\end{equation*}
for $n$ large enough.
\end{proof}

\begin{proof}[{\bf Proof of Theorem \ref{811t1}}]
Given any $0<\sigma<T$, we first consider the uniform upper bound
on Hausdorff dimensions for supp \,$X(t)$ at all times $t\in [\sigma, T)$. We adapt the same idea as
the proof of Theorem \ref{916co1} to find a cover for the support
at any time $t\in[\sigma,T)$.
Since we have a sharper estimate for $N_{n,k}$ under Condition A, for $n$ large enough,
(\ref{121018e3}) in the proof of Theorem \ref{916co1} can be
replaced by
\begin{equation*}
\text{supp}\,X\(t\)\subseteq\bigcup_{1\leq i<2^{n/\al}n^{2/\al}}\mathbb B\(X_{i}\(\(k-1\)2^{-n}-\),2Ch\(2^{-n}\)\)
\end{equation*}
for any $t\in[k2^{-n},(k+1)2^{-n}\wedge T)$ and $1\leq k\leq 2^nT$, i.e., for any $t\in[\sigma,T)\subseteq[2^{-n},T)$, supp\,$X(t)$ is contained in
at most $\lfloor2^{n/\al}n^{2/\al}\rfloor$ closed balls each of
which has a radius bounded from above by $2Ch\(2^{-n}\)$.

For any $\epsilon>0$ we have
\begin{equation*}
\begin{split}
%&\lim_{n\rightarrow\infty}\sup_{\sigma\leq t\leq T}N_{n,\lfloor2^nt\rfloor}\left[2Ch\(2^{-n}\)\right]^{\frac{2+\epsilon}{\al}}\\
%\leq
\lim_{n\rightarrow\infty}{\lfloor2^{\frac{n}{\al}}n^{\frac{2}{\al}}\rfloor}\left(2Ch\(2^{-n}\)\right)^{\frac{2+\epsilon}{\al}}
\leq&\lim_{n\rightarrow\infty}(2C)^{\frac{2+\epsilon}{\al}}2^{\frac{n}{\al}}n^{\frac{2}{\al}}\(h\(2^{-n}\)\)^{\frac{2+\epsilon}{\al}}\\
=&\lim_{n\rightarrow\infty}\(2C\)^{\frac{2+\epsilon}{\al}}\(\log2\)^{\frac{2+\epsilon}{2\al}}2^{-\frac{n\epsilon}{2\al}}n^{\frac{6+\epsilon}{2\al}}\\
=&0,
\end{split}
\end{equation*}
which implies $\mathcal{H}^{\frac{2+\epsilon}{\al}}(\text{supp}\,
X(t))=0 $. Since $\epsilon$ is arbitrary, the Hausdorff dimensions
for supp\,$X(t)$ at all times $t\in[\sigma,T)$ are uniformly bounded from above by $2/\al$.
% for all times in $[\sigma,T)$.

Finally, let $\sigma\equiv1/T$ and $T\rightarrow\infty$. The
Hausdorff dimension for supp\,$X(t)$ has uniform upper bound $2/\al$
at all positive times simultaneously.
\end{proof}

\begin{proof}[{\bf Proof of Theorem \ref{range}}]
Given any $0<\delta<T$, we also follow the proof of Theorem \ref{916co1} to find a finite cover for $\mathcal{R}\([\delta,T)\)$.
Choose $n$ large enough such that $2^{-n}\leq\theta\wedge\delta$ and (\ref{824e1}) holds. Similarly as (\ref{121018e4}) in the proof of Theorem \ref{916co1}, we have
\begin{equation*}
\begin{split}
\mathcal{R}\(\left[\delta,T\right)\)&\subseteq\mathcal{R}\(\left[2^{-n},T\right)\)\\
&\subseteq\bigcup_{1\leq k\leq 2^nT}\mathcal{R}\(\left[k2^{-n},\(k+1\)2^{-n}\wedge T\right)\)\\
&\subseteq\bigcup_{1\leq k\leq2^nT}\bigcup_{1\leq i<2^{n/\al}n^{2/\al}}\mathbb B\(X_{i}\(\(k-1\)2^{-n}-\),2Ch\(2^{-n}\)\),
\end{split}
\end{equation*}
which implies that ${\mathcal{R}\(\left[\delta,T\right)\)}$ is contained in at most $\lfloor 2^nT\rfloor\times \lfloor2^{n/\al}n^{2/\al}\rfloor$ closed balls, each of which has radius bounded from above by $2Ch(2^{-n})$.

For any $\epsilon>0$, it follows that
\begin{equation*}
\begin{split}
&\lim_{n\rightarrow\infty}\lfloor 2^nT\rfloor\times\left\lfloor 2^{\frac{n}{\al}}n^{\frac{2}{\al}}\right\rfloor\(2Ch(2^{-n})\)^{\frac{2}{\al}+2+\epsilon}\\
\leq& C(T,d,\al,\epsilon)\lim_{n\rightarrow\infty} 2^{-\frac{n\epsilon}{2}}n^{\frac{3}{\al}+1+\frac{\epsilon}{2}}
=0.
\end{split}
\end{equation*}
Since $\epsilon$ is arbitrary, the Hausdorff dimension for the range ${\mathcal{R}\(\left[\delta,T\right)\)}$ is bounded from above by ${2}/{\al}+2$.
\end{proof}

\begin{proof}[{\bf Proof of Corollary \ref{121013}}]
With initial value $\delta_0$, applying Theorem \ref{modulus2}, it is clear that almost surely
$$\mathcal{R}\([0,2^{-n})\)\subseteq\mathbb{B}\(0,Ch(2^{-n})\)$$
for $n$ large enough. From the proof of Theorem \ref{range}, we have
\begin{equation*}
\begin{split}
\mathcal{R}\(\left[0,T\right)\)
\subseteq&\mathcal{R}\([0,2^{-n})\)\bigcup\mathcal{R}\(\left[2^{-n},T\right)\)\\
\subseteq&\mathbb{B}\(0,Ch(2^{-n})\)\bigcup\bigcup_{1\leq k\leq2^nT}\bigcup_{1\leq i<2^{n/\al}n^{2/\al}}\mathbb B\(X_{i}\(\(k-1\)2^{-n}-\),2Ch\(2^{-n}\)\)
\end{split}
\end{equation*}
for $n$ large enough.

Therefore, ${\mathcal{R}\(\left[0,T\right)\)}$ is contained in at most $\lfloor 2^nT\rfloor\times \lfloor2^{n/\al}n^{2/\al}\rfloor+1$ closed balls, each of which has radius bounded from above by $2Ch(2^{-n})$.

For any $\epsilon>0$, we have
\begin{equation*}
\begin{split}
&\lim_{n\rightarrow\infty}\(\lfloor 2^nT\rfloor\times \left\lfloor2^{\frac{n}{\al}}n^{\frac{2}{\al}}\right\rfloor+1\)\(2Ch(2^{-n})\)^{\frac{2}{\al}+2+\epsilon}
%\leq& C(T,d,\al,\epsilon)\lim_{n\rightarrow\infty} 2^{-\frac{n\epsilon}{2}}n^{\frac{3}{\al}+1+\frac{\epsilon}{2}}
=0.
\end{split}
\end{equation*}
Since $\epsilon$ is arbitrary, the Hausdorff dimension for the range ${\mathcal{R}\(\left[0,T\right)\)}$ is bounded from above by ${2}/{\al}+2$.
\end{proof}

\begin{proof}[{\bf Proof of Proposition \ref{prop_mod_uni}}]
Let $\{t_i\}$ be any dense subset of $[0,T]$. Combining the proofs
for Theorem \ref{th1} and Theorem \ref{modulus2},  there exist
$\theta\equiv \theta(T,d,\alpha)<e^{-1}$ and $C\equiv C(d,\alpha)$
such that $\mathbb{P}$-a.s.
\[\text{supp} X(t_i+\Delta t)\subseteq \mathbb{B}(\text{supp} X(t_i), Ch(\Delta t))\]
for all $i$ and $0<\Delta t\leq \theta\wedge (T-t_i) $. Then for any
$t\in [0,T)$, there exists a subsequence $(t_{i_j})$ with
$t_{i_j}\downarrow t$  such that given any $n>0$,
\begin{equation*}
\begin{split}
\text{supp} X(t+\Delta t)&=\text{supp} X\(t_{i_j}+\Delta t-\(t_{i_j}-t\)\)\\
&\subseteq \mathbb{B}(\text{supp} X(t_{i_j}), Ch(\Delta t))\\
&\subseteq \mathbb{B}(\mathcal{R}\([t, t+1/n)\), Ch(\Delta t))
\end{split}
\end{equation*}
for $0<\Delta t\leq \theta\wedge (T-t)$ and $j$ large enough. So,
\[ \text{supp}\,X(t+\Delta t)\subseteq \mathbb{B}(S_t, Ch(\Delta t))\]
since $n$ is arbitrary.
\end{proof}

%%%%%%%%%%%%%%%%%%%%%%%%%%%%%%%%%%%%%%%%%%%%%%%%%%%%%%%%%%%%%%%%%%%%%%%%%%%%%%%%%%%%%%
%%%%%%%%%%%%%%%%%%%%%%%%%%%%%%%%%%%%%%%%%%%%%%%%%%%%%%%%%%%%%%%%%%%%%%%%%%%%%%%%%%%%%%%
%%%%%%%%%%%%%%%%%%%%%%%%%%%%%%%%%%%%%%%%%%%%%%%%%%%%%%%%%%%%%%%%%%%%%%%%%%%%%%%%%%%%%%%%
%\bigskip\bigskip

\end{document}